\documentclass[11pt,a4paper]{amsart}
\usepackage[DIV14,paper=a4,BCOR14mm,headinclude]{typearea}
\usepackage[british]{babel}
\usepackage{amssymb}
\usepackage{graphicx}
\usepackage{epstopdf}
\usepackage{color}
\usepackage{tikz}
\usepackage{diagbox}
\usepackage{pgfplots}
\usepackage{enumitem}
\usepackage{hyperref}
\usepackage{fp-eval}
\usepackage{esint}
\usepackage{subfigure}
\usepackage{caption}
\usepackage{mathabx}
\usepackage{booktabs}
\usepackage{todonotes}
\usepackage{subdepth}
\definecolor{darkblue}{rgb}{0,0,.7}
\hypersetup{colorlinks=true,linkcolor=darkblue,citecolor=darkblue}

\newlist{alphenum}{enumerate}{1}
\setlist[alphenum]{fullwidth,label={(\alph*)}}

\allowdisplaybreaks[0]
\theoremstyle{definition}
\newtheorem{theorem}{Theorem}[section]

\newtheorem{remark}[theorem]{Remark}
\newtheorem{lemma}[theorem]{Lemma}

\newtheorem{algorithm}[theorem]{Algorithm}

\numberwithin{figure}{section}
\numberwithin{table}{section}
\numberwithin{equation}{section}

\newcommand{\vecb}[1]{\boldsymbol{#1}}
\newcommand{\jump}[1]{[\![{#1}]\!]}

\renewcommand{\div}{\mathrm{div}}
\renewcommand{\hat}[1]{{\widehat{#1}}}

\newcommand{\mesh}{\mathcal{T}}
\newcommand{\facets}{\mathcal{F}}
\newcommand{\rr}{\mathbb{R}}
\newcommand{\cM}{{c_M}}

\newcommand{\uhdg}{{\vecb{u}_h^{hdg}}}
\newcommand{\uhdiv}{{\vecb{u}_h^{div}}}

\newcommand{\varrhohdg}{{\varrho_h^{hdg}}}
\newcommand{\varrhodiv}{{\varrho_h^{div}}}

\begin{document}
\date{\today}

\title[]{Gradient-robust hybrid DG discretizations for the
compressible Stokes equations}

\author{P.L.~Lederer}
\address{Department of Applied Mathematics, University of Twente,
Hallenweg 19, 7522NH  Enschede, Netherlands}
\email{p.l.lederer@utwente.nl}

\author{C.~Merdon}
\address{
Weierstrass Institute for Applied Analysis and
Stochastics,\\ Mohrenstr. 39, 10117 Berlin, Germany}
\email{christian.merdon@wias-berlin.de}

\maketitle

\begin{abstract}
  This paper studies two hybrid discontinuous Galerkin (HDG) discretizations for the velocity-density formulation of the compressible Stokes
  equations with respect to several desired structural properties, namely provable convergence, the preservation of
  non-negativity and mass constraints for the density, and gradient-robustness.
  The later property dramatically enhances the accuracy in well-balanced situations, such as the hydrostatic balance where
  the pressure gradient balances the gravity force.
  One of the studied schemes employs an $H(\mathrm{div})$-conforming velocity ansatz space which ensures all mentioned
  properties, while a fully discontinuous method is shown to satisfy all properties but the gradient-robustness.
  Also higher-order schemes for both variants are presented and compared in three numerical benchmark problems. 
  The final example shows the importance also for non-hydrostatic well-balanced states
  for the compressible Navier--Stokes equations.

\end{abstract}

\section{Introduction}
For incompressible flows the concept of pressure-robustness
characterizes discretizations that allow for a priori velocity error estimates
that are independent of the pressure and the viscosity parameter.
Otherwise the scheme can suffer from a severe locking phenomenon
\cite{JLMNR:sirev,cmame:linke:merdon:2016,MR3780790,GLS:2019}. A lack of pressure-robustness
can be avoided by using divergence-free schemes, e.g., \cite{SV:1983, FN:2013, GN:2014, 
guzman:et:al:2017, MR2304270, MR3833698, MR3826676, JLMR2022},
or, alternatively, non pressure-robust classical discretizations can be 'repaired' by applying
$H(\mathrm{div})$-conforming reconstruction operators at critical spots
\cite{Linke:2012, LMT:2016, cmame:linke:merdon:2016, LLMS:2017}.

Here, we consider the non-conservative form of the compressible Stokes model problem that
seeks a velocity $\vecb{u}$ and a non-negative density $\varrho$
with a mass constraint such that
\begin{align}\label{eqn:intro_modelproblem}
        -\nu \Delta \vecb{u} + \nabla p(\varrho) & = \varrho \vecb{g} + \vecb{f},\\
        \mathrm{div}(\varrho \vecb{u})
        & = 0,
\end{align}
for a given equation of state, e.g.\ the ideal gas law $p(\varrho) = \cM \varrho$.
Here, $\nu$ is the viscosity and $c_M$ is a constant related to the
inverse of the squared Mach number and the (assumed constant) temperature.
In \cite{akbas2020} the authors,
inspired by \cite{gal-09-conv} and \cite{MR3749377}, extended the
concept of pressure-robustness to the compressible Stokes equations
and connected it with the concept of well-balanced schemes. As in the
incompressible case, dominant gradient fields in the momentum balance
can appear, and methods that do not suffer from this are coined
gradient-robust (since the gradient force could be also balanced by
$\nabla(\mathrm{div}\vecb{u})$ and not only by the pressure). In the
compressible setting, well-balanced states beyond $\nabla p(\varrho) =
\nabla \vecb{f}$ can appear, in particular $\nabla p(\varrho) =
\varrho \vecb{g}$ in presence of the gravity term or other
conservative forces. More well-balanced and non-hydrostatic states are
possible, e.g., when including the convection term
$\mathrm{div}(\varrho \vecb{u} \otimes \vecb{u})$ or the geostrophic
balance in presence of the Coriolis force $2 \varrho (\Omega \times
\vecb{u})$. It is non-trivial for numerical schemes to preserve these
states accurately.

There are several approaches in the literature to design well-balanced schemes, 
mostly in the context of hyperbolic conservation laws and model problems
like the shallow water equations with bottom topography or the Euler equations with gravity,
see e.g.\ \cite{doi:10.1137/0733001,doi:10.1137/S1064827503431090,MICHELDANSAC2016568,CiCP-30-666,GROSHEINTZLAVAL2020109805}
and references therein.
A popular approach in these references is a certain modification of the source term
based on a hydrostatic reconstruction, i.e.\ a transformation to a set of variables
that stays constant in the well-balanced state. An equivalent
strategy from \cite{berberich2020high}
requires a sufficiently accurate representation of the well-balanced state
and then computes the deviations from this state.

In \cite{akbas2020} the gradient-robustness property
of a scheme was identified as one important ingredient for well-balancedness
on general meshes.
To do so, an inf-sup stable Bernardi--Raugel finite element method was coupled 
with a finite-volume method for the continuity equation. Moreover, a reconstruction
operator that preserves the discrete divergence of the test function was employed in the
gravity term,
ensuring that the discretely divergence-free part of the solution is really divergence-free
and therefore orthogonal onto gradient forces. This was the key ingredient to ensure
gradient-robustness and therefore a certain well-balancedness.
The scheme also ensures the non-negativity constraint for the density and guaranteed convergence
and is asymptotic-preserving in the sense that it converges
to a pressure-robust scheme for the incompressible Stokes equations
if the Mach number goes to zero or, equivalently, if $\cM$ goes to infinity.
In \cite{MaoXue23} an unconditional error estimate for the pressure of that scheme
for the semi-stationary compressible Stokes problem
and a similar discrete scheme was shown. However, here some additional stabilization
terms in the continuity equation were added which unfortunately
compromise the gradient-robustness.

As for pressure-robust and divergence-free methods for incompressible
flows, the concept of gradient-robustness is based on discrete exact
sequences or De Rham complexes, which ensure the structure-preserving
features of the method. Another identical concept is the
framework of compatible ($H(\mathrm{div})$-conforming) FEM, see e.g.\
\cite{cotter:thuburn:2014} where it is applied to the Euler equations
and shallow water equations. In the present paper,
hybrid discontinuous Galerkin schemes are explored that avoid the
introduction of a reconstruction operator as in \cite{akbas2020} and
straightforwardly allow for higher order schemes. Note, that an
extension of the model problem \eqref{eqn:intro_modelproblem} to a
model with the full elasticity tensor $-\nu \mathrm{div} (\mathbb{C}
\varepsilon (\vecb{u})) $ is straightforward and requires an
additional Korn inequality to hold. For the discontinuous Galerkin
methods discussed here the necessary estimates can be found in
\cite{BrennerKorn2014}, or in the context of mixed FEM for linear
elasticity \cite{https://doi.org/10.48550/arxiv.2206.14610}.

Two variants of the hybrid discontinuous Galerkin (HDG) methods are
studied and their lowest order versions are shown to converge and
preserve non-negativity and mass constraints on general meshes.
Although the main line of arguments is similar to \cite{akbas2020,
gal-09-conv}, adaptations to the DG context are needed. Moreover a
sharper stability estimate with respect to the gravity force
$\vecb{g}$ is provided.

The
first variant discretizes the velocity field in an
$H(\mathrm{div})$-conforming Brezzi--Douglas--Marini (BDM) space, which
allows that the discretely divergence-free part of the velocity is
exactly divergence-free. Therefore, it is perfectly orthogonal on any
gradient in the momentum balance. Thus, no
$H(\mathrm{div})$-conforming interpolation as in \cite{akbas2020} is needed,
but requires a different discretization of the diffusive term instead.
In the following let $q, \psi$ be smooth scalar fields. By
$L^2$-orthogonality of divergence-free velocity fields and gradients,
$\vecb{f} = \nabla q$ yields a well-balanced discrete solution with
$\vecb{u}=\vecb{0}$. Also a gravity-related balanced state with $\varrho
\vecb{g} = \varrho \nabla \psi = \nabla q$ is approximated much better
than without gradient-robustness. However, the discretization of
$\varrho$ by $\varrho_h$ generates a small perturbation $(\varrho -
\varrho_h) \vecb{g}$ that may not be fully irrotational and therefore
may cause an imbalance and spurious oscillations that scale with
$1/(c_M\nu)$. This was also observed in \cite{akbas2020}. The second
variant of the HDG method also relaxes the
$H(\mathrm{div})$-conformity and therefore the divergence-constraint
of the velocity is formulated in the spirit of the DG versions from
\cite{di2011mathematical}.  While this also allows for a provably
converging and non-negativity-preserving scheme, the relaxation of the
divergence-constraint compromises the gradient-robustness and, in
consequence, also the well-balancedness.
Numerical examples confirm in which situations the $H(\mathrm{div})$-conforming
scheme is superior, namely for low Mach numbers and small $\nu$.
Moreover, the last example demonstrates the importance of gradient-robustness
for non-hydrostatic well-balanced states like
$\varrho \mathrm{div}(\vecb{u} \otimes \vecb{u}) + \nabla p(\varrho) = \vecb{0}$
in the compressible Navier--Stokes setting, where the convection term
can be a gradient.


\medskip
The rest of the paper is structured as follows. Section~\ref{sec:Preliminaries}
introduces the model problem and basic notation and concepts. Section~\ref{sec:grobustHDG}
introduces the gradient-robust HDG scheme. Section~\ref{sec:stability_existence}
proves stability and existence of discrete solutions. Section~\ref{sec:convergence}
shows convergence of the gradient-robust scheme. Section~\ref{sec:fullyDG}
shortly discusses the fully discontinuous variant and the necessary modifications
to the stability and convergence proof. Section~\ref{sec:numerics} compares
both variants in three numerical examples with a focus on the benefits of
gradient-robustness.

\section{Preliminaries}\label{sec:Preliminaries}
In the following and for the rest of this work we consider a Lipschitz
domain $\Omega \subset \mathbb{R}^d$ with $d = 2$ or $d=3$. For a
subset $\omega \subseteq \Omega$ we use $( \bullet, \bullet)_\omega$
to denote the $L^2$ inner product on~$\omega$, with $\| \bullet
\|^2_\omega = ( \bullet, \bullet)_\omega$. For $\omega = \Omega$ we
omit the subscript, i.e. use  $( \bullet, \bullet)_\Omega =( \bullet, \bullet)$ and $\| \bullet \|_\Omega = \| \bullet \|$. We
employ standard notation of Sobolev spaces and use bold symbols for
their vector valued versions, e.g. $H^1(\Omega)$ and
$\vecb{H}^1(\Omega) = [H^1(\Omega)]^d$ for the first order Sobolev
spaces in one and $d$ dimensions, respectively. Moreover, we use the
common notation $H(\mathrm{div}, \Omega)$ (i.e. without a bold symbol)
to denote the (vector-valued) Sobolev space of functions whose
weak-divergence is in $L^2$. Finally note that we make use of a zero
index to denote a vanishing trace on $\partial \Omega$ of the
corresponding (continuous) trace operator. 

\subsection{The compressible Stokes model problem}
Let $\vecb{g} \in \vecb{L}^\infty(\Omega)$ be a given gravity force, and
additionally, for conceptual purposes, consider a second force $\vecb{f} \in
\vecb{L}^2(\Omega)$.

The weak formulation of the compressible Stokes equations seeks
$\vecb{u} \in \vecb{V} := \vecb{H}^1_0(\Omega)$
and $\varrho \in Q := L^2(\Omega)$ such that
\begin{subequations}\label{eqn:continuous_problem}
\begin{align}
    a(\vecb{u}, \vecb{v}) + b(\varrho, \vecb{v}) 
    & = (\varrho \vecb{v}, \vecb{g}) + (\vecb{v}, \vecb{f})
    && \text{for all } \vecb{v} \in \vecb{V}, \label{eqn:continuous_problem_a}\\
    c(\varrho, \vecb{u}, \lambda) 
    & = 0
    && \text{for all } \lambda \in W^{1, \infty}(\Omega),\label{eqn:continuous_problem_b}
\end{align}
\end{subequations}
where
\begin{align*}
    a(\vecb{u},\vecb{v}) & := \nu (\nabla \vecb{u}, \nabla \vecb{v}),\\
    b(\varrho, \vecb{v}) & := -(p(\varrho), \div (\vecb{v})),\\
    c(\varrho, \vecb{v}, \lambda) & := (\varrho \vecb{v}, \nabla \lambda).
\end{align*}
Throughout the paper we assume a linear  equation of state and a mass constraint given by
\begin{align} \label{eq::stateandmass}
    p(\varrho) = \cM \varrho, \quad \textrm{and} \quad (\varrho,1) = M,
\end{align}
where $M > 0$ and $\cM$ are constants. The later can be considered as
the squared inverse of the Mach number, i.e., $\cM \approx
{M\!a}^{-2}$.

\subsection{Gradient forces and hydrostatic/well-balanced solutions}
\label{sec:gradient_forces}
This section is concerned with a proper characterization of
gradient-robustness and well-balancedness.
Both concepts are related to gradient fields in the momentum balance.

For the incompressible Stokes problem one observes that any gradient force
$\vecb{f} = \nabla q$, with a given potential $q$, leads to a hydrostatic solution $\vecb{u} \equiv \vecb{0}$
and a pressure $\nabla p = \nabla q$ that fully balances $\vecb{f}$.
A numerical method that preserves this was coined pressure-robust \cite{cmame:linke:merdon:2016,JLMNR:sirev}.
The correct balancing of the gradient force $\nabla q$ exploits the $L^2$-orthogonality of
divergence-free functions on $\nabla q$, i.e.,
\begin{align*}
    (\nabla q, \vecb{v}) = - (q, \mathrm{div} (\vecb{v})) = 0, \qquad \text{for all } \vecb{v} \in \vecb{V}_0 := \lbrace \vecb{v} \in \vecb{V} : \mathrm{div} (\vecb{v} )= 0 \rbrace.
\end{align*}

In the present compressible setting given by \eqref{eqn:continuous_problem}
a similar hydrostatic balance is possible, namely
\begin{align}\label{eq:hydrostatic_balance}
    \nabla p(\varrho) = \varrho \vecb{g} = \varrho \nabla \psi.
\end{align}
This situation appears, e.g., in an atmosphere-at-rest-scenario
and might be considered equivalent to the lake-at-rest scenario in
shallow water equations with bottom topography \cite{MICHELDANSAC2016568}.
A discrete scheme that correctly balances gradient forces $\psi$ and computes
hydrostatic solutions with $\vecb{u} = \vecb{0}$ in these cases is called well-balanced.

For the equation of state $p(\varrho) = \cM \varrho$ the hydrostatic balance
can be reformulated to
\begin{align}\label{eq:hydrostatic_balance_isothermal}
    \nabla p(\varrho) = \cM \varrho \nabla (\log \varrho) = \varrho \vecb{g} = \varrho \nabla \psi.
\end{align}
This yields (uniformly positive) solutions of the form
$\varrho := \varrho_0 \exp(\psi/\cM)$ where the constant $\varrho_0$ is chosen such that the mass constraint
is satisfied.

It is non-trivial for a discrete scheme to compute hydrostatic solutions in this case
without using a priori information.
Indeed, one could subtract the exact solution
from the equation and compute a deviation density, in the spirit of, e.g., \cite{berberich2020high}.
However, in more complex situations, e.g. other forces, multi-physics or
boundary conditions or different equations of state $p(\varrho)$, analytical solutions might
be unavailable. Hence, here we are interested in an out-of-the-box scheme that is as accurate as possible without
a priori modifications.

The purpose of the forcing $\vecb{f}$ in \eqref{eqn:continuous_problem} is
to better explain the importance of gradient-robustness as an important
ingredient for well-balancedness. To this end let us consider a gradient force
$\vecb{f} = \nabla q$ and the hydrostatic balance
\begin{align}\label{eq:hydrostatic_balance2}
    \nabla p(\varrho) = \vecb{f} = \nabla q.
\end{align}
Due to the non-negativity and mass constraint, this balance (and therefore
a hydrostatic solution) is only satisfied if one can choose a constant
$C$ such that the density is given by 
\begin{align*}
    \varrho := (q - C)/\cM,
\end{align*}
and at the same time stays non-negative and satisfies $((q - C)/\cM, 1) = M$.
This is only possible if $q$ is small enough or $M$ is large enough and
such forces $\vecb{f}$ are called admissible, see also \cite[Lemma 4.3]{akbas2020}
for a motivation.

\smallskip
To summarize we consider these two qualities of well-balancedness for this model problem:
\begin{itemize}
\item a scheme for \eqref{eqn:continuous_problem} is said to be gradient-robust if it admits a hydrostatic solution whenever
$\vecb{g} = \vecb{0}$ and $\vecb{f}$ is an admissible gradient force;
\item a scheme for \eqref{eqn:continuous_problem}
is said to be well-balanced if it admits a hydrostatic solution whenever $\vecb{f} = \vecb{0}$ and $\vecb{g}$
is a gradient force.
\end{itemize}

\begin{remark}
    If $\varrho$ is the exact solution for given $\vecb{g} = \nabla \psi$ and $M > 0$, then $\vecb{f} := \varrho \nabla \psi$
    is always an admissible force. This follows by the calculation above. A comparison of the
    two force terms implies that the lack of well-balancedness of
    a gradient-robust scheme is therefore caused by or determined by the non-irrotational part of $(\varrho - \varrho_h) \nabla \psi$,
    where $\varrho_h$ is the density approximation of the scheme, whereas the irrotational part
    of that quantity is treated correctly by a gradient-robust scheme.
\end{remark}

As established in \cite{akbas2020}, gradient-robustness needs a correct balancing of
divergence-free forces and gradient forces via structural properties of
$H(\mathrm{div})$-conforming
finite element spaces, namely the $L^2$-orthogonality of divergence-free functions
and gradients like the force terms discussed above.

\section{A Gradient-robust HDG scheme}\label{sec:grobustHDG}
This section discusses a hybrid discontinuous Galerkin (HDG) discretization for the
compressible Stokes equation where the discrete velocity $\vecb{u}_h$ is
$H(\mathrm{div})$-conforming. This implies
gradient-robustness and asymptotic convergence
to a pressure-robust discretization of the incompressible Stokes problem
when the Mach number tends to zero $M\!a \rightarrow 0$ .

\subsection{Notation}
Consider a regular triangulation $\mathcal{T}$ of $\Omega$ into
simplices. The set of vertices is given by $\mathcal{N}$ and the sets
of faces by $\mathcal{F}$. For simplification we further assume that
$\mathcal{T}$ is quasi uniform and use $h_T :=
\operatorname{diam}(T)$, $h_F := \operatorname{diam}(F)$ for $T \in
\mathcal{T}$ and $F \in \mathcal{F}$ and define $h:=\max\limits_{T \in
\mathcal{T}} h_T$. Due to quasi uniformity we have $h \approx h_T
\approx h_F$. 
 
On a face $F \in \facets$ we define a unit vector normal $\vecb n_F$ with an
arbitrary but fixed orientation. Note, that the (fixed) orientation of
$\vecb n_F$ also defines the orientation of the jump operator
$\jump{\cdot}$, e.g. let $F = T_1 \cap T_2$ for two elements $T_1$ and
$T_2$ and fix $\vecb n_F$ to point from $T_1$ to $T_2$, then we have on $F$
\begin{align*}
    \jump{q} := q_{h,1} - q_{h,2}, \quad \textrm{with} \quad q_{h,i} := q_h|_{T_i}, \textrm{ for } i = 1,2.
\end{align*}
On the domain boundary and on boundaries $\partial T$ for $T \in
\mathcal{T}$ we use $\vecb n$ to denote the outward pointing normal
vector. Further, we define the tangential projection for a function
$\vecb v$ by $\vecb v_t := \vecb v - (\vecb v \cdot \vecb n) \vecb n$.
Further, on the domain boundary we have $\vecb n$ = $\vecb n_F$, and
the jump operator equals the identity.  We denote by
$P^k(\omega)$ the set of polynomials on $\omega \subset \Omega$ of
total order $k$, and again use bold symbols to denote the corresponding vector-valued versions.


\subsection{The $H(\mathrm{div})$-HDG scheme}

Consider the finite element spaces
\begin{align*}
    \vecb{V}_h &:= \vecb{P}_{k}(\mesh) \cap H_0(\div, \Omega),\\
    \widehat{\vecb{V}}_h &:= \left\{\widehat{\vecb{v}}_h \in \vecb{L}^2(\facets): \widehat{\vecb{v}}_h|_F \in \vecb{P}_{k}(F), \widehat{\vecb{v}}_h \cdot \vecb{n}_F = 0 ~\forall F \in \facets,  \widehat{\vecb{v}}_h = 0 \textrm{ on } \partial \Omega \right\},\\
    Q_h &:= P_{k-1}(\mesh).
\end{align*}
Here, $\vecb{V}_h$ is the $H(\mathrm{div})$-conforming BDM space of order $k$ which is used as velocity ansatz space, and 
$Q_h$ is the discontinuous ansatz space for the pressure and
density. The hybridization space $\widehat{\vecb{V}}_h$ is used to couple the discontinuous tangential parts of discrete velocities in $\vecb{V}_h$ in a hybrid DG fashion. Thus, $\widehat{\vecb{V}}_h$ can be seen as as the ansatz space for the tangential
traces of velocities on the skeleton $\mathcal{F}$.

The suggested discrete scheme seeks $\left((\vecb{u}_h, \widehat{\vecb{u}}_h), \varrho_h\right)
\in (\vecb{V}_h \times \widehat{\vecb{V}}_h) \times Q_h$ such that
\begin{subequations}\label{eqn:discrete_scheme}
\begin{align}
    a_h((\vecb{u}_h, \widehat{\vecb{u}}_h),(\vecb{v}_h, \widehat{\vecb{v}}_h)) 
    + b(p(\varrho_h), \vecb{v}_h)
     & = F_h(\vecb{v}_h) + G_h(\varrho_h, \vecb{v}_h), \label{eqn:discrete_scheme_abd}\\
     c_h(\varrho_h, \vecb{u}_h, \lambda_h) & = 0, \label{eqn:discrete_scheme_c}\\
     (\varrho_h,1) &= M, \label{eqn:discrete_scheme_e}
\end{align}
\end{subequations}
for all $(\vecb{v}_h, \widehat{\vecb{v}}_h) \in \vecb{V}_h \times
\widehat{\vecb{V}}_h$ and $\lambda_h \in Q_h$. Here, the forms are
defined by
\begin{align*}
    a_h((\vecb{u}_h, \widehat{\vecb{u}}_h),(\vecb{v}_h, \widehat{\vecb{v}}_h)) & := 
    \nu \sum_{T \in \mathcal{T}} (\nabla \vecb{u}_h, \nabla \vecb{v}_h)_T 
    +(\nabla \vecb{u}_h  \vecb{n}, (\widehat{\vecb{v}}_h - \vecb{v}_h)_t)_{\partial T}\\
    &\hspace{1cm} + (\nabla \vecb{v}_h \vecb{n}, (\widehat{\vecb{u}}_h - \vecb{u}_h)_t)_{\partial T} 
    + \frac{\alpha k^2}{h} ((\widehat{\vecb{u}}_h - \vecb{u}_h)_t, (\widehat{\vecb{v}}_h - \vecb{v}_h)_t)_{\partial T}\\
    c_h(\varrho_h, \vecb{u}_h, \lambda_h) 
    &:=-\sum_{T \in \mathcal{T}} (\varrho_h \vecb{u}_h, \nabla \lambda_h)_T
    +  (\vecb{u}_h\cdot \vecb{n} \varrho_h^{up}, \lambda_h)_{\partial T}, \\
    G_h(\varrho_h, \vecb{v}_h) & :=
    (\vecb{g}, \varrho_h \vecb{v}_h),\\
    F_h(\vecb{v}_h) & :=
    (\vecb{f}, \vecb{v}_h),
\end{align*}
where $\varrho_h^{up}$ is the standard upwind value associated to
$\vecb{u}_h$. The vanishing Dirichlet value of the velocity was
incorporated in an essential (direct) manner for the normal part
via $\vecb{V}_h$, and in a weak (DG-like) manner for the tangential
part via $\widehat{\vecb{V}}_h$.
As usual for (hybrid) DG methods, the parameter $\alpha
> 0$ has to be chosen sufficiently large enough. In our numerical
examples we always choose $\alpha = 10$. For a comparison and
discussion on DG and HDG methods we refer to the literature, e.g.
\cite{zbMATH05837531, cockburn2009unified, L_MTH_2010}. Note that an
alternative formulation of the diffusive fluxes would be possible by
means of a mixed formulation which is inherently stable (without
choosing an $\alpha>0$). See \cite{MR4637970} for more details and a
comparison with respect to the stabilization parameter. For the sake of simplicity we do
not use this formulation in this work. 
The proposed scheme of this work is similar to the one in \cite{akbas2020}, but
instead of an $H^1$-conforming Bernardi--Raugel method a hybrid
$H(\div)$-DG approach is employed. A stability and convergence proof
is given in Sections~\ref{sec:stability_existence} and
\ref{sec:convergence}.


The analysis involves the usual HDG norm
\begin{align} \label{eq::hdgnorm}
    \| (\vecb{u}_h, \hat {\vecb u}_h) \|^2_{1,h} 
    &:= \sum\limits_{T \in \mesh} \| \nabla \vecb u_h \|_T^2 + \frac{1}{h_T} \| (\vecb u_h - \hat{\vecb u}_h)_t \|_{\partial T}^2.
\end{align}
Moreover, we recall that the chosen velocity and pressure spaces allow for
an LBB-condition with respect to the space
\begin{align*}
    Q_h \cap L^2_0(\Omega) := \Bigl\{\lambda \in L^2(\Omega): \int_\Omega \lambda = 0 \Bigr\},
\end{align*}
which can be found for example in \cite{MR3867384}. This is one
important ingredient for the stability analysis of the scheme.

\begin{lemma}[LBB-stability Stokes]\label{lem:LBBstabilityStokes}
    There exists some $\beta > 0$ independent of $h$,  such that for
    all $\lambda_h \in Q_h \cap
    L^2_0(\Omega)$, we have
    \begin{align*}
        \sup_{\substack{(\vecb{v}_h, \hat{\vecb{v}}_h) \in {\vecb V}_h \times \hat{\vecb V}_h\\ (\vecb{v}_h,  \hat{\vecb v}_h) \neq (\vecb{0}, \vecb{0})}}
        \frac{b(\lambda_h, \vecb{v}_h)}{\| (\vecb{v}_h, \hat{ \vecb{v}}_h) \|_{1,h} } \geq \beta \| \lambda_h \|.
    \end{align*}
\end{lemma}
 
\begin{lemma}[Gradient-robustness]
    The scheme \eqref{eqn:discrete_scheme} is gradient-robust.
    
    \begin{proof}
        Consider an admissable gradient force $\vecb{f} = \nabla q$ (in the sense of
        Section~\ref{sec:gradient_forces}) and the
        incompressible Stokes problem: seek $(\vecb{u}^0_h, \widehat{\vecb{u}}^0_h)$ with
        $\mathrm{div} (\vecb{u}^0_h) = 0$ and (up to an arbitrary global constant $C$) some $p_h \in Q_h$ such that
        \begin{align}\label{eqn:incompressible_stokes_subproblem}
            a_h((\vecb{u}^0_h, \widehat{\vecb{u}}^0_h),(\vecb{v}_h, \widehat{\vecb{v}}_h))
            - (p_h, \mathrm{div} (\vecb{v}_h)) & = (\nabla q, \vecb{v}_h) = -(q, \mathrm{div} (\vecb{v}_h)).
        \end{align}
        Testing with $(\vecb{v}_h, \widehat{\vecb{v}}_h) = (\vecb{u}^0_h, \widehat{\vecb{u}}^0_h)$
        yields
        \begin{align*}
            \| (\vecb{u}^0_h, \widehat{\vecb{u}}^0_h) \|^2_{1,h} = 0.
        \end{align*}
        Standard Stokes pressure estimates show $p_h := \Pi_{Q_h} q +
        C$, where $\Pi_{Q_h}$ is the $L^2$-projection onto $Q_h$. Due
        to the admissibility assumption of $\vecb{f}$ we can then find a
        constant $C$ such that $\varrho_h := p_h + C > 0$ also
        satisfies the mass constraint. Hence, $(\vecb{u}_h,
        \widehat{\vecb{u}}_h) = (\vecb{u}^0_h, \widehat{\vecb{u}}^0_h)
        = (\vecb{0},\vecb{0}) $ and $\varrho_h$ also solves
        \eqref{eqn:discrete_scheme}, thus the scheme admits a
        hydrostatic solution.
    \end{proof}
\end{lemma}

\section{Stability and existence of solutions}\label{sec:stability_existence}
This section shows stability and existence of solutions as well as positivity
and mass preservation of $\varrho_h$ for the suggested HDG scheme.
However, only the lowest order case $k = 1$ guarantees the positivity preservation
of $\varrho_h$. 

\subsection{Stability}
The following Lemma recalls \cite[Lemma 4.2]{akbas2020}.

\begin{lemma}\label{lem:upwinding_estimate} Let $\vecb{u} \in
    L^2(\mathcal{F})$ be a single-valued function on each face of the
    triangulation, and let $\varrho^{up}_h$ be the $\vecb{u}$-associated
    upwind value. For any twice continuously differentiable convex
    function $\phi: [0, \infty) \rightarrow \rr^{+}$ there holds 
    \begin{multline} \label{upwind_estimate}
        \sum_{T \in \mesh} \int_{\partial T} \vecb{u} \cdot \vecb{n} \varrho_h^{up} \phi^\prime(\varrho_h)
        - \sum_{T \in \mesh} \int_{\partial T} \vecb{u} \cdot \vecb{n} (\varrho_h \phi^\prime(\varrho_h)) - \phi(\varrho_h)) \\
        =
        \frac{1}{2} \sum_{F \in \facets}  \phi^{\prime \prime}(\varrho^{F}_h) \Big| \int_F \vecb{u} \cdot \vecb{n}_F \jump{\varrho_h}^2 \Big| \ge 0,
    \end{multline}
    with intermediate values $\varrho_h^{F} \in
    [\min(\varrho_{h,1},\varrho_{h,2}),\max(\varrho_{h,1},\varrho_{h,2})]$
    where $\varrho_{h,1}$ and $\varrho_{h,2}$ are the restrictions of
    $\varrho$ on the adjacent elements of facet $F$. 
\end{lemma}
\begin{proof}
    First note that the left difference can be written as
    \begin{multline*}
        \sum_{T \in \mesh} \int_{\partial T} \vecb{u} \cdot \vecb{n} \varrho_h^{up} \phi^\prime(\varrho_h)
        - \sum_{T \in \mesh} \int_{\partial T} \vecb{u} \cdot \vecb{n} (\varrho_h \phi^\prime(\varrho_h)) - \phi(\varrho_h)) \\
        =  \sum_{T \in \mesh} \int_{\partial T} \vecb{u} \cdot \vecb{n} (\varrho_h^{up} \phi^\prime(\varrho_h) - \varrho_h \phi^\prime(\varrho_h)) + \phi(\varrho_h)).
    \end{multline*}
    Let $T_1$ and $T_2$ be the two neighboring elements of a
    facet $F \in \facets$, then we can rewrite the sum above as 
    \begin{multline*}
        \sum_{T \in \mesh} \int_{\partial T} \vecb{u} \cdot \vecb{n} (\varrho_h^{up} \phi^\prime(\varrho_h) - \varrho_h \phi^\prime(\varrho_h)) + \phi(\varrho_h)) \\
        =  \sum_{F \in \facets} \int_{F} \vecb{u} \cdot \vecb{n}_F 
        (\phi^\prime(\varrho_{h,1}) (\varrho_h^{up} - \varrho_{h,1}) + \phi^\prime(\varrho_{h,2}) (\varrho_{h,2} - \varrho_h^{up})
        + \phi(\varrho_{h,1}) - \phi(\varrho_{h,2})).
    \end{multline*}
    Let $\theta_F = (\phi^\prime(\varrho_{h,1}) (\varrho_h^{up} -
    \varrho_{h,1}) + \phi^\prime(\varrho_{h,2}) (\varrho_{h,2} - \varrho_h^{up}) +
    \phi(\varrho_{h,1}) - \phi(\varrho_{h,2}))$. Consider the case $\int_F\vecb{u}
    \cdot \vecb{n}_F \ge 0$, then $\varrho_h^{up} - \varrho_{h,1}$, and thus
    with the Taylor expansion at $\varrho_{h,2}$, i.e.
    \begin{align*}
        \phi(\varrho_{h,1}) = \phi(\varrho_{h,2}) + \phi^\prime(\varrho_{h,2}) (\phi(\varrho_{h,1}) - \phi(\varrho_{h,2}))
        + \frac{1}{2}\phi^{\prime\prime}(\varrho_h^F) (\phi(\varrho_{h,1}) - \phi(\varrho_{h,2}))^2,
    \end{align*}
    with a point $\varrho_h^F \in [\min(\varrho_{h,1},\varrho_{h,2}),\max(\varrho_{h,1},\varrho_{h,2})]$, we get
    \begin{align*}
        \theta_F = \frac{1}{2}\phi^{\prime\prime}(\varrho_h^F) (\phi(\varrho_{h,1}) - \phi(\varrho_{h,2}))^2 \ge 0,
    \end{align*}
    where the non-negativity follows by the convexity of $\phi$. For
    the other case, i.e. $\int_F\vecb{u}
    \cdot \vecb{n}_F < 0$, we derive similarly
    \begin{align*}
        \theta_F = -\frac{1}{2}\phi^{\prime\prime}(\varrho_h^F) (\phi(\varrho_{h,1}) - \phi(\varrho_{h,2}))^2 \le 0,
    \end{align*}
    and thus we get 
    $
        \sum_{F \in \facets} \int_{F} \vecb{u} \cdot \vecb{n}_F \theta_F \ge 0.
    $
\end{proof}

\begin{theorem}[Stability]\label{thm:stability}
    For the solution of \eqref{eqn:discrete_scheme}, it holds
\begin{align}
  \| (\vecb{u}_h, \hat{ \vecb{u}}_h) \|_{1,h} & \lesssim \| \vecb{f} \|  + \| \vecb{g} \|_{L^\infty (\Omega)} \| \varrho_h \|,\label{stability_estimate_u}\\
  \sum_{F \in \facets}  \phi^{\prime \prime}(\varrho^{F}_h) \Big| \int_F \vecb{u}_h \cdot \vecb{n}_F \jump{\varrho_h}^2 \Big|  & \lesssim 
  (\| \vecb{f} \|  + \| \vecb{g} \|_{L^\infty(\Omega)} \| \varrho_h \|)^2,\label{stability_estimate_upw}\\
  \Big(1 - \frac{C}{ \cM^2}  \| \vecb{g} \|_{L^\infty(\Omega)}^2\Big) \| \varrho_h \|^2 & \lesssim \frac{1}{ \cM^2} \| \vecb{f} \|^2 + M^2.\label{stability_estimate_rho}
\end{align}
Stability for $\varrho_h$ and $\vecb{u}_h$ is therefore guaranteed for
 $ \| \vecb{g} \|_{L^\infty(\Omega)} / \cM$ small enough. The constant $C$ is a generic
 constant that depends on the shape of the cells, $\lvert \Omega \rvert$ and $\alpha$,
 but not on $h$, $\cM$, $M$ or $\nu$.
\end{theorem}
\begin{proof}
    Following \cite{MR2304270,L_MTH_2010} we have the coercivity
    estimate
    \begin{align} \label{eq::coercivity}
        \nu \| (\vecb{u}_h, \hat{ \vecb{u}}_h) \|^2_{1,h} \lesssim \nu a((\vecb{u}_h, \hat{ \vecb{u}}_h),(\vecb{u}_h, \hat{ \vecb{u}}_h)).
    \end{align}
    Testing the momentum equation
    with $\vecb{v}_h = \vecb{u}_h$ gives
    \begin{align*}
        \nu a((\vecb{u}_h, \hat{ \vecb{u}}_h),(\vecb{u}_h, \hat{ \vecb{u}}_h)) +
         b(p(\varrho_h), \vecb{u}_h) = F_h(\vecb{u}_h) + G_h(\varrho_h,\vecb{u}_h).
    \end{align*}
    Due to the upwinding we get the
    correct sign from Lemma~\ref{lem:upwinding_estimate}. For this
    choose the convex function $\phi(s) = \cM s \log(s)$, with
    $\phi^\prime(s) = \cM(\log(s) + 1)$, then we have $\varrho_h \phi^\prime(\varrho_h) -
    \phi(\varrho_h) = c\varrho_h = p(\varrho_h)$. By that \eqref{upwind_estimate} reads as
    \begin{align*}
        c_h(\varrho_h, \vecb{u}_h, \cM(1 + \log(\varrho_h))) + b(p(\varrho_h), \vecb{u}_h) = 
        \frac{1}{2} \sum_{F \in \facets}  \phi^{\prime \prime}(\varrho^{F}_h) \Big| \int_F \vecb{u}_h \cdot \vecb{n}_F \jump{\varrho_h}^2 \Big|.
    \end{align*}
With \eqref{eqn:discrete_scheme_c} and $\cM(1 + \log(\varrho_h)) \in
Q_h$, we get $ b(p(\varrho_h), \vecb{u}_h) \geq 0$. It remains
to bound the right-hand side.
    Using a discrete Friedrichs-type inequality, see for example
    \cite{Brenner:PFDG}, we get
    \begin{align*}
        F_h(\vecb{u}_h)
        \lesssim \| \vecb{f} \| \| (\vecb{u}_h, \hat{ \vecb{u}}_h) \|_{1,h}.
    \end{align*}
    For the other right-hand side term we get similarly
    \begin{align*}
        G_h(\varrho_h,\vecb{u}_h) = (\vecb{g}, \varrho_h \vecb{u}_h)
        & \leq \| \vecb{g} \|_{L^\infty(\Omega)}   \| \varrho_h \| \| \vecb{u}_h \|
         \leq \| \vecb{g} \|_{L^\infty(\Omega)}   \| \varrho_h \| \| (\vecb{u}_h, \hat{ \vecb{u}}_h) \|_{1,h},
    \end{align*}
thus we conclude $\| (\vecb{u}_h, \hat{ \vecb{u}}_h) \|_{1,h} \lesssim
  \| \vecb{f} \| + \| \vecb{g} \|_{L^\infty(\Omega)}  \| \varrho_h \|$. Note, that
  by the above construction we have also proven \eqref{stability_estimate_upw}. \\



For the proof of \eqref{stability_estimate_rho} let $p_h := p(\varrho_h)$ and define the mean value
$\overline p_h := | \Omega |^{-1} \int_\Omega p_h$.
By Lemma \ref{lem:LBBstabilityStokes} it exists a
$\vecb{v}_h \in \vecb{V}_h$ with $\mathrm{div} \vecb{v}_h = p_h - \overline p_h$ such that 
\begin{align*}
    \| p_h - \overline p_h \|^2 
    = b(p_h - \overline p_h, \vecb{v}_h),
\end{align*}
and $\| (\vecb{v}_h, \widehat{\vecb{v}}_h) \|_{1,h}\lesssim \| p_h - \overline{p}_h \|$.
Hence
\begin{align*}
  \| p_h - \overline{p}_h \|^2
&=b(p_h - \overline p_h, \vecb{v}_h) \\
  & = -a_h((\vecb{u}_h, \widehat{\vecb{u}}_h),(\vecb{v}_h, \widehat{\vecb{v}}_h)) - F_h(\vecb{v}_h) - G_h(\varrho_h, \vecb{v}_h)\\
  & \lesssim \| (\vecb{u}_h, \widehat{\vecb{u}}_h ) \|_{1,h} \| (\vecb{v}_h, \widehat{\vecb{v}}_h ) \|_{1,h} + \| \vecb{f} \| \| \vecb{v}_h \| + \lvert G_h(\varrho_h, \vecb{v}_h) \rvert\\
  & \lesssim \left( \| \vecb{f} \| + \| \vecb{g} \|_{L^\infty(\Omega)}  \| \varrho_h \|\right) \| p_h - \overline{p}_h \| .
\end{align*}
where we again used a discrete Friedrichs inequality in the last step. \\

The mass constraint \eqref{eqn:discrete_scheme_e} yields the identity
\begin{align*}
    \overline{p}_h = \frac{1}{|\Omega|} \int_\Omega p_h 
    = \frac{1}{|\Omega|} \int_\Omega \cM \varrho_h
    = \cM M \lvert \Omega \rvert^{-1}
    \quad \Rightarrow \quad
    \| \overline{p}_h \|^2 = \cM^2 M^2 \lvert \Omega \rvert^{-1}.
\end{align*}
Eventually, a Pythagoras theorem and the previous estimates yield
\begin{align*}
    \cM^2 \| \varrho_h \|^2 = \| p_h \|^2
    = \| p_h-\overline p_h \|^2 + \| \overline p_h \|^2
    \leq C \left(\| \vecb{f} \|^2 + \| \vecb{g} \|_{L^\infty(\Omega)} ^2 \| \varrho_h \|^2 \right) + \cM^2 M^2 \lvert \Omega \rvert^{-1},
\end{align*}
which can be reordered into
\begin{align*}
(1 - C \cM^{-2} \| \vecb{g} \|_{L^\infty(\Omega)} ^2) \| \varrho_h \|^2 \lesssim \cM^{-2} \| \vecb{f} \|^2 + M^2.
\end{align*}
This concludes the proof.
\end{proof}

\subsection{Existence of discrete solutions}

This section suggests a fixed-point iteration for the computation of a
solution of \eqref{eqn:discrete_scheme} and shows existence of at least
one fixed-point. The steps are very similar to \cite{akbas2020}.
Throughout this section we assume the lowest order case $k=1$
to guarantee that all computed densities stay non-negative.

\begin{algorithm}[Fixed-point algorithm]\label{alg:fixedpointiteration}
    Given a triangulation $\mathcal{T}$ and a step size $\tau > 0$
    and initial values $\vecb{u}_h^0 = \vecb{0}$ and $\varrho_h^0 := M / \lvert \Omega \rvert$,
    compute, for $n = 0,1,2,\ldots$ until satisfied, 
    \begin{align*}
        (\vecb{u}^{n+1}_h, \widehat{\vecb{u}}^{n+1}_h, \varrho^{n+1}_h)
        = F((\vecb{u}^{n}_h, \widehat{\vecb{u}}^{n}_h, \varrho^{n}_h)),
    \end{align*}
    where $F : \vecb{V}_h \times \widehat{\vecb{V}}_h \times Q_h \rightarrow \vecb{V}_h \times \widehat{\vecb{V}}_h \times Q_h$
    denotes the fixed-point mapping that computes the new iterate by the following sub-systems.
    The new velocity iterate $(\vecb{u}^{n+1}_h, \widehat{\vecb{u}}^{n+1}_h) \in \vecb{V}_h \times \widehat{\vecb{V}}_h$
    satisfies, for all $(\vecb{v}_h, \widehat{\vecb{v}}_h) \in \vecb{V}_h \times \widehat{\vecb{V}}_h$,
    \begin{align}\label{eqn:subsystem_u}
    \nu a_h((\vecb{u}^{n+1}_h, \widehat{\vecb{u}}^{n+1}_h),(\vecb{v}_h, \widehat{\vecb{v}}_h))
     & = F_h(\vecb{v}_h) + G_h(\varrho_h^{n}, \vecb{v}_h) - b(p(\varrho^{n}_h), \vecb{v}_h),
    \end{align}
    and the new density iterate $\varrho_h^{n+1} \in Q_h$ satisfies
    \begin{align}\label{eqn:subsystem_rho}
        \tau^{-1} (\varrho_h^{n+1}, \lambda_h) + c_h(\varrho^{n+1}_h, \vecb{u}^{n+1}_h, \lambda_h)
        = \tau^{-1} (\varrho^{n}_h, \lambda_h)
        \quad \text{for all } \lambda_h \in Q_h.
    \end{align}
    The iteration is stopped if the residuals of both sub-systems are
    below some given tolerance.
\end{algorithm}

\begin{lemma}[Solvability of the sub-systems] \label{lem::std_solvability}
    Both sub-systems \eqref{eqn:subsystem_u} and \eqref{eqn:subsystem_rho}
    are solvable. Moreover, if $\rho_h^{n} > 0$ and $(\rho^n_h,1) = M$,
    then also $\rho_h^{n+1} > 0$ and $(\rho^{n+1}_h,1) = M$.

\begin{proof}
    The solvability of the update \eqref{eqn:subsystem_u} for
  $\vecb{u}^{n+1}_h$ follows from the coercivity of $a_h$, see
  \eqref{eq::coercivity}.

  The solvability of the update \eqref{eqn:subsystem_rho} for
  $\varrho^{n+1}_h$ follows from the fact that for $k=1$ (i.e.
  $\varrho$ is approximated by piecewise constants) the system matrix
  is an $M$-matrix. Indeed, the representation matrix for the form
  $c_h(\varrho_h^{n+1}, \vecb{u}_h^{n+1}, \lambda_h)$ (for fixed
  $\vecb{u}_h^{n+1})$ is weakly diagonal-dominant, has non-negative
  diagonal entries and non-positive off-diagonal entries, and has zero
  row-sums. Hence, adding a positive definite diagonal matrix yields
  an $M$-matrix. That matrix is invertible and has only positive
  entries. Hence, the positivity of the previous density iterate
  $\varrho_h^n$ is preserved. Moreover, also the mass constraint is
  preserved which follows from testing with $\lambda_h \equiv 1$.
\end{proof}
\end{lemma}

The following lemma establishes existence of a fixed-point via Brouwer's fixed-point theorem.

\begin{lemma}[Existence of solutions]
    On every fixed shape-regular mesh $\mathcal{T}$,
    the discrete nonlinear system \eqref{eqn:discrete_scheme} has at least one solution.
\end{lemma}
\begin{proof}
  The mapping $F$ that defines the fixed-point iteration in Algorithm
  \ref{alg:fixedpointiteration} is linear and continuous, since it
  consists of the composition of two solvable linear systems of
  equations, see Lemma~\ref{lem::std_solvability}.
  
  To apply Brouwer's fixed-point theorem, it remains to show that $F$
  maps a convex set into itself. This can be shown by similar
  arguments as in Theorem~\ref{thm:stability}, but the term
  $b(p(\varrho^{n}_h), \vecb{v}_h)$ has to be estimated by 
  \begin{align*}
    b(p(\varrho^{n}_h), \vecb{v}_h)
    \leq \| \mathrm{div} (\vecb{v}_h) \| \| p(\varrho^{n}_h) \|
    \leq \cM  \| (\vecb{v}_h, \hat{\vecb{v}}_h) \|_{1,h} \| \varrho^{n}_h \|.
  \end{align*}
  Since all discrete norms on the fixed triangulation $\mathcal{T}$
  are equivalent (with some possibly mesh-dependent constant $C(h)$)
  and the mass constraint is preserved in every iteration,
  we can employ the pessimistic but sufficient bound
  \begin{align*}
    \| \varrho^{n}_h \| \leq C(h) \| \varrho^{n}_h \|_{L^1} = C(h) M.
  \end{align*}
  Hence, all iterates stay within a bounded convex set, which justifies
  the application of Brouwer's fixed-point theorem to conclude the existence
  of a fixed-point.
\end{proof}

\section{Convergence of the scheme}\label{sec:convergence}
This section shows convergence of the discrete solutions
to a weak solution of the model problem under suitable assumptions.

For this we apply a Rellich-type theorem of \cite{MR2931412} for (H)DG
approximations. Although \cite{MR2931412} considers only the scalar
case, the vector valued case follows accordingly. The result involves an
element-wise lifting operator, defined on each $T \in \mathcal{T}$ by
\begin{align*}
    R_h|_T : \vecb{L}^2(\partial T) \rightarrow [P_k(T)]^{d \times d}, \quad   \vecb{g} \mapsto \Phi_h,
\end{align*}
where $\Phi_h$ is given by
\begin{align*}
    (\Phi_h, \Psi_h)_T = (\vecb{g}, \Psi_h \vecb{n}) \quad \text{for all } \quad \Psi_h \in [P_k(T)]^{d \times d}.
\end{align*}
Moreover, there is the operator $S_h : \vecb{H}^1(\mathcal{T}) \rightarrow \prod_{T \in \mathcal{T}} \vecb{L}^2(\partial T)$
defined by
\begin{align*}
    S_h (\vecb{u}) := \bigl\lbrace \left(\vecb{u}|_T\right)|_{\partial T} \bigr\rbrace_{T \in \mathcal{T}},
\end{align*}
that collects all cell boundary traces.

\begin{theorem}\label{thm:convergence}
    Consider a sequence of shape-regular triangulations $(\mathcal{T}_h)_{h \rightarrow 0}$.
    Let $((\vecb{u}_h, \widehat{\vecb{u}}_h), \varrho_h)
    $ denote the corresponding discrete solution of \eqref{eqn:discrete_scheme} on $\mathcal{T}_h$. Then, up to extraction of a
    subsequence, it holds
    \begin{itemize}
        \item[(i)] the sequence $(\vecb{u}_h)_{h \rightarrow 0}$ converges strongly to some $\vecb{u} \in \vecb{L}^2(\Omega) \cap \vecb{H}^1_0(\Omega)$ and 
        $\nabla \vecb{u}_h + R_h(\widehat{\vecb u}_h + (S_h \vecb{u}_h)_t) \rightharpoonup \nabla \vecb{u}$,
        \item[(ii)] the sequence $(\varrho_h)_{h \rightarrow 0}$ converges weakly in $L^2(\Omega)$ to a limit $\varrho \in L^2(\Omega)$,
        \item[(iii)] the sequence $(p_h)_{h \rightarrow 0} := (p(\varrho_h))_{h \rightarrow 0}$ converges weakly in $L^2(\Omega)$ to a limit $p_\star \in L^2(\Omega)$, 
        \item[(iv)]  $p_\star$ and $\varrho$ satisfy the equation of state, i.e., $p_\star = p(\varrho)$, 
        \item[(v)] the limit $(\vecb{u}, \varrho)$ is a weak solution of \eqref{eqn:continuous_problem}.
    \end{itemize}
\end{theorem}
\begin{proof}[Proof of (i)-(iii)]
    By Theorem~\ref{thm:stability} the sequence $\vecb{u}_h$ is bounded and the result follows
    from \cite[Theorem~1]{MR2931412}.
\end{proof}
\begin{proof}[Proof of (iv)]
    For $\gamma = 1$ this is straightforward, since the equation of
    state is linear. To see this consider a function $\varphi \in
    C_c^\infty(\Omega)$ and some sequence $\varphi_h := \Pi_{Q_h}
    \varphi$ that converges strongly towards $\varphi$. For that
    sequence, due to weak-strong convergence, it holds
    \begin{align*}
        (p_h, \varphi_k) & \rightarrow (p_\star, \varphi),
    \end{align*}
    and on the other hand
    \begin{align*}
        (p_h, \varphi_k) = (p(\varrho_k), \varphi_k) = (c \varrho_k, \varphi_k) & \rightarrow (c \varrho, \varphi) = (p(\varrho), \varphi).
    \end{align*}
    This allows to conclude
    \begin{align*}
        (p_\star - p(\varrho), \varphi) = 0 \quad \text{for all } C_c^\infty(\Omega),
    \end{align*}
    which implies (iv).
\end{proof}
\begin{proof}[Proof of (v)]
    We first prove that $(\vecb{u}_h, \varrho_h)$ satisfy the momentum equation.
    Take any vector-valued smooth test function $\vecb{v} \in
    \vecb{C}_c^\infty(\Omega)$ and approximate it by best-approximations
    $\vecb{v}_h \in \vecb{V}_h \! \cap \! \vecb{H}^1_0(\Omega)$ such that 
    \begin{align*}
        \vecb{v}_h \rightarrow \vecb{v} \quad \text{strong in } \vecb{H}_0^1(\Omega).
    \end{align*}
    Now choose $\widehat{\vecb{v}}_h = (\vecb{v}_h)_t$ then
    strong-weak convergence, see \cite{MR2931412}, yields
    \begin{align*}
    a_h((\vecb{u}_h, \widehat{\vecb{u}}_h),(\vecb{v}_h, \widehat{\vecb{v}}_h))
    & := 
    \sum_{T \in \mathcal{T}} (\nabla \vecb{u}_h, \nabla \vecb{v}_h)_T
    + (\nabla \vecb{v}_h \vecb{n}, (\widehat{\vecb{u}}_h - \vecb{u}_h)_t)_{\partial T} \\
    & = \sum_{T \in \mathcal{T}} (\nabla \vecb{u}_h + R_h|_T((\widehat{\vecb{u}}_h - S_h \vecb{u}_h)_t), \nabla \vecb{v}_h)_T\\
    & \rightarrow (\nabla \vecb{u}, \nabla \vecb{v}) = a(\vecb{u},\vecb{v}).
    \end{align*}
    Next, since $\vecb{v}_h$ is continuous, we have
    \begin{align*}
        b_h(p(\varrho_h), \vecb{v}_h) 
        & = \sum_{T \in \mathcal{T}} - (\mathrm{div}(\vecb{v}_h), p(\varrho_h))_T = -\int_\Omega \div (\vecb{v}_h) p(\varrho_h) 
         \rightarrow -\int_\Omega \div (\vecb{v}) p(\varrho).
    \end{align*}
    Using strong convergence and weak-strong convergence one also obtains
    \begin{align*}
        \int_\Omega \vecb{f} \vecb{v}_h \rightarrow \int_\Omega \vecb{f} \vecb{v},
        \quad \text{and} \quad
        \int_\Omega \vecb{g} \varrho_h \vecb{v}_h \rightarrow \int_\Omega \vecb{g} \varrho \vecb{v}.
    \end{align*}
    
    It remains to prove that the limits fulfill the continuity equation.
    Consider a test function $\psi \in
    \vecb{C}^\infty(\Omega)$ and approximate it by best-approximations
    $\psi_h \in P_1(\mathcal{T}) \cap \vecb{H}^1(\Omega)$ such that 
    \begin{align*}
        \psi_h \rightarrow \psi \quad \text{strong in } \vecb{H}^1(\Omega)
        \quad \text{and} \quad
        \| \nabla \psi_h \|_{L^\infty(\Omega)}  \lesssim \| \nabla \psi \|_{L^\infty(\Omega)} . 
    \end{align*}
    The bound follows from an inverse inequality for polynomials and
    the stability of the $H^1$-best-approximation, i.e.
    \begin{align*}
        \| \nabla \psi_h \|_{L^\infty(\Omega)}  \lesssim h^{-d/2} \| \nabla \psi_h \| \lesssim h^{-d/2} \| \nabla \psi \|,
    \end{align*}
    and the estimate
    \begin{align*}
        \| \nabla \psi \|^2 \leq \|\nabla \psi\|_{L^1} \|\nabla \psi\|_{L^\infty(\Omega)}  \leq \|1\| \|\nabla \psi\| \|\nabla \psi\|_{L^\infty(\Omega)}  \approx h^{d/2} \|\nabla \psi\| \|\nabla \psi\|_{L^\infty(\Omega)} .
    \end{align*}
    
    The discrete momentum $\varrho_h \vecb{u}_h$ is approximated into
    some $q_h \in \vecb{V}_h \subset \vecb{H}(\mathrm{div}, \Omega)$ 
    by
\begin{equation}\label{eq:HdivInterpolation_rhou}
  \vecb{q}_h|_T := I_h^{\mathrm{RT}_0} (\varrho^{\text{up}}_h \vecb{u}_h|_T)
  \quad \text{on each } T \in \mathcal{T},
\end{equation}
where $I_h^{\mathrm{RT}_0}$ is the interpolation operator into the
lowest order Raviart--Thomas space, see \cite{brezzi}. Note, that this
interpolation is divergence-free, because
\begin{align*}
    \mathrm{div}(\vecb{q}_h|_T) = c_h(\varrho_h, \vecb{u}_h, \chi_T),
\end{align*}
where $\chi_T \in Q_h$ is the indicator function of $T$, i.e. $\chi_T = 1$ on $T$ and zero elsewhere.
With that, it holds
\begin{equation*}
  0  = (\psi_h, \mathrm{div} \, \vecb{q}_h) \\
     = - (\vecb{q}_h, \nabla \psi_h) \\
     = -( \vecb{q}_h - \varrho_h \vecb{u}_h, \nabla \psi_h)
    - (\varrho_h \vecb{u}_h, \nabla \psi_h).
\end{equation*}
It remains to show that the first term on the right-hand side
converges to zero. A triangle inequality yields
\begin{equation}\label{eq:critical_estimate}
\begin{split}
  \left | ( \vecb{q}_h - \varrho_h \vecb{u}_h, \nabla \psi_h) \right |
    & \leq  \sum_{T \in \mathcal{T}} \left | \nabla \psi_h|_T \cdot \left ( \int_T (\vecb{q}_h - \varrho_h I_h^{\mathrm{RT}_0} \vecb{u}_h) \,  + \int_T \varrho_h (I_h^{\mathrm{RT}_0} \vecb{u}_h - \vecb{u}_h) \,   \right ) \right | \\
    & \leq
        C \sum_{T \in \mathcal{T}} \| \vecb{q}_h - \varrho_h I_h^{\mathrm{RT}_0} \vecb{u}_h \|_{L^1(T)} + C \| \varrho_h \| \, \| I_h^{\mathrm{RT}_0} \vecb{u}_h
        -  \vecb{u}_h \|.
    \end{split}
\end{equation}
The term $\| \varrho_h \| \, \| I_h^{\mathrm{RT}_0} \vecb{u}_h
-  \vecb{u}_h \|$ converges to $0$, according to
the interpolation properties of $I_h^{\mathrm{RT}_0}$
and the stability estimate for $\| \varrho_h \|$ and $\| \vecb{u}_h \|_{1,h}$.
It remains to estimate $\sum_T \| \vecb{q}_h - \varrho_h I_h^{\mathrm{RT}_0} \vecb{u}_h \|_{L^1(T)}$.
Interpolation properties of $I_h^{\mathrm{RT}_0}$ yield
\begin{align*}
  \sum_T \| \vecb{q}_h - \varrho_h I_h^{\mathrm{RT}_0} \vecb{u}_h \|_{L^1(T)}
  & \lesssim \sum_T h_T \sum_{F \in \mathcal{F}(T)} \left \lvert (\varrho^\text{up}_h - \varrho_h|_T ) \int_F \vecb{u}_h \cdot \vecb{n}_F \, ds \right\rvert\\
  & \lesssim \sum_{F \in \mathcal{F}(\Omega)} h_F \lvert \jump{ \varrho_h}_F \rvert \, \left\lvert \int_F \vecb{u}_h \cdot \vecb{n}_F \right\rvert := A.
\end{align*}
There holds \(A \rightarrow 0\) which can be proven as follows. A Cauchy inequality shows
\begin{align}\label{eqn:estimateA_case1}
  A \lesssim \left( \sum_{F \in \mathcal{F}(\Omega)}  \left\lvert \int_F \vecb{u}_h \cdot \vecb{n}_F \right\rvert \left(\varrho_h^F\right)^{-1} \jump{ \varrho_h}^2_F \right)^{1/2} \left(\sum_{F \in \mathcal{F}(\Omega)} h_F^2 \left\lvert \int_F \vecb{u}_h \cdot \vecb{n}_F \right\rvert \varrho_h^F \right)^{1/2}.
\end{align}
The left sum is bounded by Theorem~\ref{thm:stability}. To show that
the second sum converges to zero, we employ a H\"older inequality,
a trace inequality and an inverse inequality on some
neighboring simplex \(T_F\) of \(F\) to obtain
\begin{align}\label{eqn:bound_for_uF}
    \left\lvert \int_F \vecb{u}_h \cdot \vecb{n}_F \right\rvert
  \lesssim \| \vecb{u}_h \|^{1/2}_{T_F} \| \nabla \vecb{u}_h \|^{1/2}_{T_F} \| 1 \|_{F}
  \lesssim h_F^{(d-2)/2} \| \vecb{u}_h \|_{T_F}.
\end{align}
Hence,
\begin{align*}
    \left(\sum_{F \in \mathcal{F}(\Omega)} h_F^2 \left\lvert \int_F \vecb{u}_h \cdot \vecb{n}_F \right\rvert \varrho_h^F \right)^{1/2}
    \lesssim \left(\sum_{F \in \mathcal{F}(\Omega)} h_F^{(d+2)/2} \| \vecb{u}_h \|_{T_F} \varrho_h^F\right)^{1/2}.
\end{align*}
Then, another Cauchy inequality, a Friedrichs inequality
for piecewise $H^1$ functions \cite{Brenner:PFDG} and some overlap arguments yield
\begin{align*}
  \left(\sum_{F \in \mathcal{F}(\Omega)} h_F^2 \left\lvert \int_F \vecb{u}_h \cdot \vecb{n}_F \right\rvert  \varrho_h^F \right)^{1/2}
  & \leq \left(\sum_{F \in \mathcal{F}(\Omega)} \| \vecb{u}_h\|_{T_F}^2 \right)^{1/4} \left( \sum_{F \in \mathcal{F}(\Omega)} h_F^{d+2} \left(\varrho_h^F\right)^2 \right)^{1/4}\\
  & \lesssim \| (\vecb{u}_h, \hat{\vecb{u}}_h) \|^{1/2}_{1,h} \left( \sum_{F \in \mathcal{F}(\Omega)} h_F^{d+2} \left(\varrho_h^F\right)^2 \right)^{1/4}.
\end{align*}
Since \(\varrho_h^F\) is smaller than \(\varrho_h|_{T_F}\) for some
neighboring simplex \(T_F\) of \(T\), we also can bound the remaining sum by
\begin{align*}
 \left( \sum_{F \in \mathcal{F}(\Omega)} h_F^{d+2} \left(\varrho_h^F\right)^2 \right)^{1/4}
 \lesssim \left( \sum_{F \in \mathcal{F}(\Omega)} h_F^2 \lvert T_F \rvert \varrho_h|_{T_F}^2 \right)^{1/4}
 \leq h^{1/2} \| \varrho_h \|^{1/2}.
\end{align*}
According to Theorem~\ref{thm:stability} the norm $\| (\vecb{u}_h, \hat{\vecb{u}}_h) \|_{1,h}$ and $\| \varrho_h \|$ are bounded and
so we eventually arrive at
  \begin{align*}
    A \lesssim h^{1/2}.
  \end{align*}
  This and weak-strong convergence ($\vecb{u}_h \nabla \psi_h$ converges strongly against $\vecb{u} \nabla \psi$)
  implies
  \begin{align*}
    \lvert (\varrho \vecb{u}, \nabla \psi) \rvert
    & = \lvert (\varrho \vecb{u}, \nabla \psi) - (\varrho_h \vecb{u}_h, \nabla \psi_h) \rvert 
      + \lvert (\varrho_h \vecb{u}_h, \nabla \psi_h) \rvert\\
    & \lesssim \lvert (\varrho \vecb{u}, \nabla \psi) - (\varrho_h \vecb{u}_h, \nabla \psi_h) \rvert + h^{1/2}
    \rightarrow 0.
  \end{align*}
  This concludes the proof.
\end{proof}

\begin{remark}[Asmyptotic convergence to a pressure-robust scheme]
On a fixed mesh and for $\cM \rightarrow \infty$,
the solutions of the scheme \ref{eqn:discrete_scheme} converge to
a pressure-robust divergence-free solution of the incompressible
Stokes equations, see \cite[Lemma 6.4]{akbas2020} for details and a proof.

\end{remark}

\section{A fully discontinuous HDG scheme}\label{sec:fullyDG}

This section elaborates on the qualitative improvements
by strictly enforcing the normal continuity of the velocity.
As for the Stokes model problem, the $H(\mathrm{div})$
conformity yields $L^2$ orthogonality of the divergence-free
part of the velocity with gradients. This property is lost
when a full HDG scheme is used that also allows jumps of the
normal component. For comparison in the numerical experiments
also this scheme shall be briefly discussed.

The ansatz spaces for this variant reads as 
\begin{align*}
    \vecb{V}_h &:= \vecb{P}_{k}(\mesh),\\
    \widehat{\vecb{V}}_h &:= \left\{\widehat{\vecb{v}}_h \in \vecb{L}^2(\facets): \widehat{\vecb{v}}_h|_F \in \vecb{P}_{k}(F), \widehat{\vecb{v}}_h = 0 \textrm{ on } \partial \Omega \right\},\\
    Q_h &:= P_{k-1}(\mesh).
\end{align*}
The full HDG scheme seeks $\left((\vecb{u}_h, \widehat{\vecb{u}}_h), (\varrho_h, \hat \varrho)\right)
\in (\vecb{V}_h \times \widehat{\vecb{V}}_h) \times Q_h$ such that
\begin{subequations}\label{eqn:discrete_scheme_fulldg}
\begin{align}
    \nu a_h((\vecb{u}_h, \widehat{\vecb{u}}_h),(\vecb{v}_h, \widehat{\vecb{v}}_h)) 
    + b_h(p(\varrho_h), \vecb{v}_h)
     & = F_h(\vecb{v}_h) + G_h(\varrho_h, \vecb{v}_h), \label{eqn:discrete_scheme_fulldg_abd}\\
     c_h(\varrho_h, \vecb{u}_h, \lambda_h) & = 0, \label{eqn:discrete_scheme_fulldg_c}\\
      (\varrho_h,1) &= M, \label{eqn:discrete_scheme_fulldg_e}
\end{align}
\end{subequations}
for all $(\vecb{v}, \widehat{\vecb{v}}_h) \in \vecb{V}_h \times \widehat{\vecb{V}}_h$
and $\lambda_h \in Q_h$.
Here, the forms are defined by
\begin{align*}
    a_h((\vecb{u}_h, \widehat{\vecb{u}}_h),(\vecb{v}_h, \widehat{\vecb{v}}_h)) & := 
    \sum_{T \in \mathcal{T}} (\nabla \vecb{u}_h, \nabla \vecb{v}_h)_T 
    + (\nabla \vecb{u}_h  \vecb{n}, (\widehat{\vecb{v}}_h - \vecb{v}_h))_{\partial T}\\
    &+ (\nabla \vecb{v}_h \vecb{n}, (\widehat{\vecb{u}}_h - \vecb{u}_h))_{\partial T} 
    + \frac{\alpha k^2}{h} ((\widehat{\vecb{u}}_h - \vecb{u}_h), (\widehat{\vecb{v}}_h - \vecb{v}_h))_{\partial T},\\
    b_h(\varrho_h, (\vecb{v}_h, \hat{\vecb{v}}_h)) & :=
    \sum_{T \in \mathcal{T}} -(\varrho_h, \mathrm{div} \vecb{v}_h)_T + 
    ((\vecb{v}_h - \hat{ \vecb{v}}_h) \cdot \vecb{n}, \varrho_h)_{\partial T},
    \\
    c_h(\varrho_h, (\vecb{u}_h,\hat{\vecb{u}}_h), \lambda_h) 
    &:=-\sum_{T \in \mathcal{T}} (\varrho_h \vecb{u}_h, \nabla \lambda_h)_T
    +  (\hat{\vecb{u}}_h \cdot \vecb{n} \varrho_h^{up}, \lambda_h)_{\partial T}, \\
    G_h(\varrho_h, \vecb{v}_h) & :=
    (\vecb{g}, \varrho_h \vecb{v}_h),\\
    F_h(\vecb{v}_h) & :=
    (\vecb{f}, \vecb{v}_h).
\end{align*}
Compared to \eqref{eqn:discrete_scheme} the missing normal-continuity causes some
changes. In particular, the upwinding term in $c_h$ now involves
$\hat{\vecb{u}}_h \cdot \vecb{n}$, which can be interpreted as a mean value
of the potentially discontinuous flux $\vecb{u}_h \cdot \vecb{n}$. 

Stability and convergence of the scheme can be shown in a similar way
as for the $H(\mathrm{div})$-conforming HDG scheme
\eqref{eqn:discrete_scheme}. Therefore, we only summarize the result
and state the main differences in the proof. Note, that the HDG-norm
now changes to 
\begin{align*} 
    \| (\vecb{u}_h, \hat {\vecb u}_h) \|^2_{1,h} 
    &:= \sum\limits_{T \in \mesh} \| \nabla \vecb u_h \|_T^2 + \frac{1}{h_T} \| (\vecb u_h - \hat{\vecb u}_h) \|_{\partial T}^2,
\end{align*}
but we use the same symbol for simplicity. 

\begin{theorem}[Stability]\label{thm:stability_fulldg}
    For the solution of \eqref{eqn:discrete_scheme_fulldg}, it holds
\begin{align*}
  \| (\vecb{u}_h, \hat{ \vecb{u}}_h) \|_{1,h} & \lesssim \| \vecb{f} \|  + \| \vecb{g} \|_{L^\infty(\Omega)}  \| \varrho_h \|,\\
  \sum_{F \in \facets}  \phi^{\prime \prime}(\varrho^{F}_h) \Big| \int_F \hat{\vecb{u}}_h \cdot \vecb{n}_F \jump{\varrho_h}^2 \Big|  & \lesssim 
  (\| \vecb{f} \|  + \| \vecb{g} \|_{L^\infty(\Omega)}  \| \varrho_h \|)^2,\\
  \Big(1 - \frac{C}{ \cM^2}  \| \vecb{g} \|_{L^\infty(\Omega)} ^2\Big) \| \varrho_h \|^2 & \lesssim \frac{1}{ \cM^2} \| \vecb{f} \|^2 + M^2,
\end{align*}
for some generic constant $C$ that depends on the shape of the cells,
 $\lvert \Omega \rvert$ and $\alpha$, but not on $h$, $\cM$, $M$ or
 $\nu$.
\end{theorem}
\begin{proof}
The main difference for the stability proof compared to the one of Theorem~\ref{thm:stability}
is that Lemma~\ref{lem:upwinding_estimate} is employed for
$\vecb{u} = \hat{\vecb{u}}_h$ and $\phi(s) = \cM s \log(s)$ which yields
\begin{align*}
    c_h(\varrho_h, \vecb{u}_h, \cM(1 + \log(\varrho_h))) 
    + b_h(\varrho_h, (\vecb{u}_h, \hat{\vecb{u}}_h)) =
    \frac{1}{2} \sum_{F \in \facets}  \phi^{\prime \prime}(\varrho^{F}_h) \Big| \int_F \hat{\vecb{u}}_h \cdot \vecb{n}_F \jump{\varrho_h}^2 \Big|.
\end{align*}
The rest of the arguments is identical.
\end{proof}

\begin{theorem}
    Consider a sequence of shape-regular triangulations $(\mathcal{T}_h)_{h \rightarrow 0}$.
    Let $(\vecb{u}_h, \widehat{\vecb{u}}_h, \varrho_h)
    $ denote the corresponding discrete solution of \eqref{eqn:discrete_scheme_fulldg} on $\mathcal{T}_h$. Then, up to extraction of a
    subsequence, it holds
    \begin{itemize}
        \item[(i)] the sequence $(\vecb{u}_h)_{h \rightarrow 0}$ converges strongly to some $\vecb{u} \in \vecb{L}^2(\Omega) \cap \vecb{H}^1_0(\Omega)$ and 
        $\nabla \vecb{u}_h + R_h(\widehat{\vecb u}_h + S_h \vecb{u}_h) \rightharpoonup \nabla \vecb{u}$,
        \item[(ii)] the sequence $(\varrho_h)_{h \rightarrow 0}$ converges weakly in $L^2(\Omega)$ to a limit $\varrho \in L^2(\Omega)$,
        \item[(iii)] the sequence $(p_h)_{h \rightarrow 0} := (p(\varrho_h))_{h \rightarrow 0}$ converges weakly in $L^2(\Omega)$ to a limit $p_\star \in L^2(\Omega)$, 
        \item[(iv)]  $p_\star$ and $\varrho$ satisfy the equation of state, i.e., $p_\star = p(\varrho)$, 
        \item[(v)] the limit $(\vecb{u}, \varrho)$ is a weak solution of \eqref{eqn:continuous_problem}.
    \end{itemize}
\end{theorem}
\begin{proof}
    Again, only the main differences compared to the proof of Theorem~\ref{thm:convergence}
    are stated.
First, the choice of \eqref{eq:HdivInterpolation_rhou} has to be altered to
\begin{equation*}
    \vecb{q}_h|_T := I_h^{\mathrm{RT}_0} (\varrho_{\text{upw}} \widehat{\vecb{u}}_h|_T)
    \quad \text{on each } T \in \mathcal{T}
  \end{equation*}
  which then again is divergence-free. In the critical estimate \eqref{eq:critical_estimate}
  one now obtains
  \begin{equation*}
    \begin{split}
      \left | ( \vecb{q}_h - \varrho_h \vecb{u}_h, \nabla \psi_h) \right |
        & \leq  \sum_T \left | \nabla \psi_h|_T \cdot \left ( \int_T (\vecb{q}_h - \varrho_h I_h^{\mathrm{RT}_0} \hat{\vecb{u}}_h) \,  + \int_T \varrho_h (I_h^{\mathrm{RT}_0} \hat{\vecb{u}}_h - \vecb{u}_h) \,   \right ) \right | \\
        & \leq
            C \sum_T \| \vecb{q}_h - \varrho_h I_h^{\mathrm{RT}_0} \hat{\vecb{u}}_h \|_{L^1(T)} + C \| \varrho_h \| \, \| I_h^{\mathrm{RT}_0} \hat{\vecb{u}}_h
            -  \vecb{u}_h \|.
        \end{split}
    \end{equation*}
    Here, the first term is treated as in the other proof and the second term can be bounded by
    \begin{align*}
        \| I_h^{\mathrm{RT}_0} \hat{\vecb{u}}_h -  \vecb{u}_h \|^2
   \lesssim \sum_F h \left(\int_F (\hat{\vecb{u}}_h -  \vecb{u}_h) \cdot \vecb{n} ds \right)^2
   \lesssim \sum_T h^3 \| (\vecb u_h - \hat{\vecb u}_h) \|_{\partial T}^2 \leq h^4 \| (\vecb u_h - \hat{\vecb u}_h) \|_{1,h}^2.
    \end{align*}
    The rest of the proof is identical.
\end{proof}

\begin{remark}[Non-gradient-robustness]
    The scheme \eqref{eqn:discrete_scheme_fulldg} is in general not gradient-robust. The reason
    is that the incompressible Stokes subproblem for the divergence-free part of the solution
    is not pressure-robust, since the integration by parts in the right-hand side of \eqref{eqn:incompressible_stokes_subproblem}
    is not possible without additional jump terms due to the relaxed $H(\mathrm{div})$-conformity.
\end{remark}

\section{Numerical examples}\label{sec:numerics}
This section studies three numerical examples to compare the two
variants with respect to the importance of gradient-robustness
and experimental convergence rates.

\subsection{Convergence rates}

In this section we want to discuss and analyze the approximation
properties and error convergence rates of our methods. For this consider
the gradient field $\nabla \Psi := (0, -y^2)^T$, and choose $\varrho$
such that it solves the hydrostatic equation \eqref{eq:hydrostatic_balance}, i.e.,
\begin{align} \label{eq::hydro_numeric}
    \cM \nabla \varrho = \varrho \nabla \Psi 
    \quad \Rightarrow \quad 
    \varrho := e^{-y^{3} / (3 \cM)} c^{-1}_\Omega, 
    \quad 
    \text{with}
    \quad
    c_\Omega = \int_\Omega e^{-y^{3}/(3 \cM) },
\end{align}
which results in a mass constraint $M = 1$. Further let $\zeta
= 100 x^2(1-x)^2y^2(1-y)^2 $ and $\vecb{u} := (-\partial_y \zeta,
\partial_x \zeta)/\varrho$, then we define the driving forces
\begin{align*}
    \vecb{g} :=  
    \nabla \Psi,
     \quad 
     \text{and}
     \quad
     \vecb{f} := - \nu \Delta \vecb{u}.
\end{align*}
By construction we then have $\div(\varrho \vecb{u}) =
\div (\partial_y \zeta, \partial_x \zeta) = 0$, thus $\varrho$ and
$\vecb{u}$ are solutions of the equations
\eqref{eqn:continuous_problem}. 

In the following we compare 
\begin{itemize}
    \item the normal continuous approximation $\uhdiv$, i.e. the solution of \eqref{eqn:discrete_scheme}, and
    \item the fully discontinuous approximation $\uhdg$, i.e. the solution of \eqref{eqn:discrete_scheme_fulldg}.
\end{itemize}
We consider the cases $\nu \in \{1, 10^{-6}\}$ and $\cM \in \{1,
100\}$. For all pairs of parameters $(\nu, \varrho)$
Figures~\ref{fig::vortex_c1_n0}-\ref{fig::vortex_c100_n6} show the
convergence history of the $L^2$ error for the velocity and the
density as well as the discrete $H^1$ error of the velocity for
polynomial orders $k \in \lbrace 1,2,3 \rbrace$. To improve the readability we simplify
the notation for the discrete $H^1$ error by
\begin{align*}
    \| \vecb{u} - \vecb{u}_h^\bullet \|_{1,h} := \| (\vecb{u}, \vecb{u}|_\facets) - (\vecb{u}_h^\bullet, \widehat{\vecb{u}}^\bullet_h)\|_{1,h}.
\end{align*}
We use an initial mesh with 96 elements and a uniform refinement.
Note that although we have only proven stability for the lowest order case,
the higher order cases were stable for all computations. 
In the following we discuss the results in more detail:

\begin{itemize}
    \item The \textbf{compressible} case $\nu = 1, \cM = 1$, Figure \ref{fig::vortex_c1_n0}: In this case we do not expect a big
    difference between the approximations $\uhdg$ and $\uhdiv$. Indeed, all errors are close to each other and we observe an optimal convergence rate.
    \item The \textbf{compressible} case $\nu = 10^{-6}, \cM = 1$,
    Figure \ref{fig::vortex_c1_n6}: As expected, the velocity errors
    still converge with optimal order but are deteriorated by the
    small viscosity. Surprisingly, though still being slightly
    shifted, the normal continuous approximation $\uhdiv$ is less
    effected. A similar observation is made for the density error.
    While still converging with optimal order it scales with the
    viscosity. In contrast to the velocity approximation no difference
    between $\varrhohdg$ and $\varrhodiv$ can be found. 
    \item The (nearly) \textbf{incompressible cases} $\nu = 1, \cM =
    100$, Figure \ref{fig::vortex_c100_n0} and $\nu = 10^{-6}, \cM =
    100$, Figure \ref{fig::vortex_c100_n6}: Since $\cM = 100 \Rightarrow
    M\!a \approx 10^{-4}$, this setting can be considered as nearly
    incompressible, thus we expect the result to behave as for a
    discretization of the incompressible Stokes equations. Indeed,
    while all errors converge with optimal order, only the normal
    continuous approximation $\uhdiv$ is pressure-robust, i.e. robust
    with respect to a small viscosity. This is in agreement with the
    literature regarding exactly divergence-free approximations, see
    for example \cite{MR3826676}.
\end{itemize}

\begin{figure}
    \includegraphics[width=\textwidth]{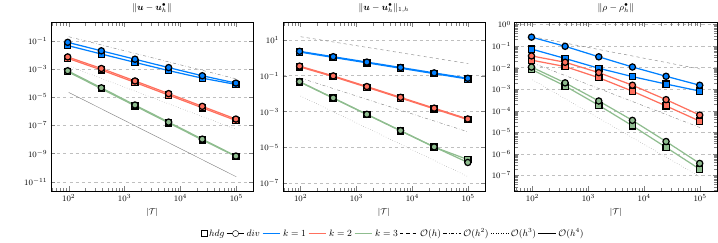}
    \caption{Convergence rates for $k \in \lbrace 1,2,3 \rbrace$ for the unit square test with $\nu = 1$ and $\cM = 1$. }
    \label{fig::vortex_c1_n0}
\end{figure}

\begin{figure}
    \includegraphics[width=\textwidth]{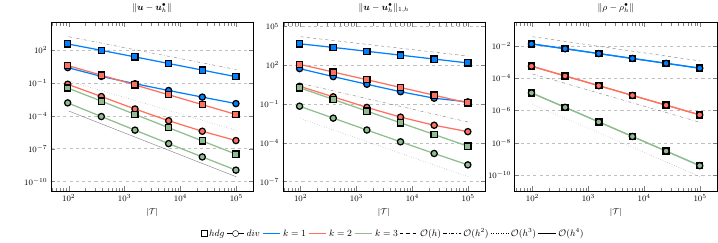}
    \caption{Convergence rates for $k \in \lbrace 1,2,3 \rbrace$ for the unit square test with $\nu = 10^{-6}$ and $\cM = 1$. }
    \label{fig::vortex_c1_n6}
\end{figure}

\begin{figure}
    \includegraphics[width=\textwidth]{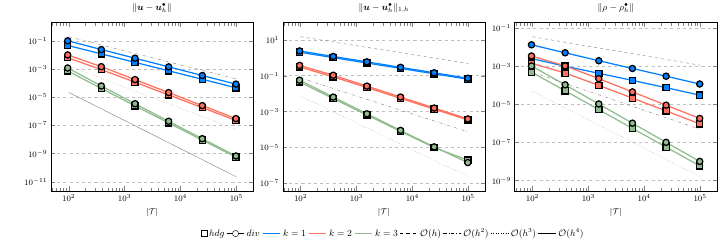}
    \caption{Convergence rates for $k \in \lbrace 1,2,3 \rbrace$ for the unit square test with $\nu = 1$ and $\cM = 100$. }
    \label{fig::vortex_c100_n0}
\end{figure}

\begin{figure}
    \includegraphics[width=\textwidth]{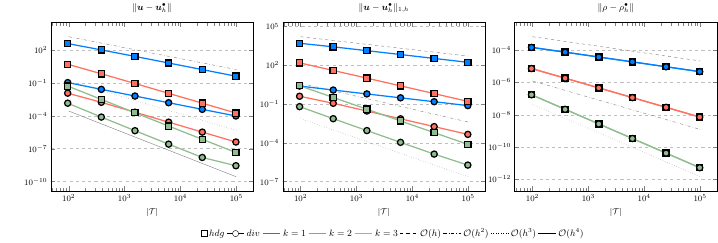}
    \caption{Convergence rates for $k \in \lbrace 1,2,3 \rbrace$ for the unit square test with $\nu = 10^{-6}$ and $\cM = 100$. }
    \label{fig::vortex_c100_n6}
\end{figure}

In contrast to the above setting one could also choose the right hand side as
\begin{align} \label{eq::allg}
    \vecb{g} :=  
    \nabla \Psi  - \frac{\nu \Delta \vecb{u}}{\varrho},
    \quad
    \text{and}
    \quad
    \vecb{f} := 0.
\end{align}
In Figure \ref{fig::vortex_c1_n6_allg} we have plotted the errors for
the case $\cM = 1$ and $\nu = 10^{-6}$. As one can see (compared to
Figure~\ref{fig::vortex_c1_n6}) there is no influence with respect to
the choice of the right hand side. For the other cases we observe the
same results, thus the plots are omitted for simplicity. Note, that
moving the gradient forces $\nabla \Psi$ also to $\vecb{f}$ results in a
big difference, see next example. 

\begin{figure}
    \includegraphics[width=\textwidth]{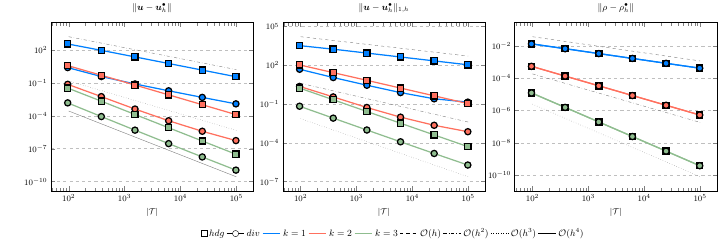}
    \caption{Convergence rates for $k \in \lbrace 1,2,3 \rbrace$ for the unit square test with $\nu = 10^{-6}$ and $c_M = 1$ and right hand side \eqref{eq::allg}.}
    \label{fig::vortex_c1_n6_allg}
\end{figure}

\subsection{Atmosphere at rest over a mountain}
 
Consider the \textit{mountain} function given by
\begin{align*}
    \mathcal{M}(y) :=  0.3 e^{\frac{- (y - 0.4)^2}{0.08^2}} + 0.2 e^{\frac{- (y - 0.6)^2}{0.1^2}},
\end{align*}
then we have the domain $\Omega := \{(x,y) \in \mathbb{R}^2: 0 < x <
1, \mathcal{M}(x) < y <1 \}$. We aim to approximate a hydrostatic
balance with the exact solution $\vecb{u} = 0$ and $\varrho$ as in the
previous example (note that $c_\Omega$ in \eqref{eq::hydro_numeric}
changes in order to get $M=1$). This gives $\vecb g = \nabla \Psi =
(0, -y^2)^T$. In Figure~\ref{fig::mountain_pic_div} and
Figure~\ref{fig::mountain_pic_hdg} we have plotted the the absolute
value of the velocity $\uhdiv, \uhdg$ and the density
$\varrhodiv,\varrhohdg$ for $\nu \in \lbrace 1, 10^{-6} \rbrace$ and $\cM = 1$ and $k =
3$. Unfortunately, as expected, the velocity error has a dependency on
the viscosity $\nu$ similarly as in the previous example, and the
normal continuous approximation gives a much better approximation. In
Figure~\ref{fig::mountain_c1_n0} and Figure~\ref{fig::mountain_c1_n6}
we give the error history for the mountain example for the same test
cases. In contrast to the previous example we did not use nested
meshes for the calculation (via a uniform refinement) of errors but
used $h_{max} = 0.3 / 2^i$ for $i = 0,\ldots,5$ for the mesh
generator. In addition we used a local mesh size $h_{loc} = 0.01$ in
order to properly capture the geometry of the mountain (as can be seen
in Figure \ref{fig::mountain_c1_n0} and Figure
\ref{fig::mountain_c1_n6}). During the mesh generation we used
$h_{loc}$ at the mountain surface if $h_{loc} < h_{max}$, and
$h_{max}$ otherwise. Thus, although the mesh is initially not quasi
uniform, it results in a quasi uniform mesh on later levels. As one
can see, this results in an initially fast pre-asymptotic decrease in
the error, while the same convergence rate as in the previous example
is observed later on. For the highest order case $k = 3$ we see a
stagnation of the error close to machine precision. Similarly as in
the previous example, for $\cM=100$ we only see an improvement for the
gradient-robust solution $\uhdiv$, which is plotted in
Figure~\ref{fig::mountain_pic_div_bigc}. Also note that one can now
clearly see that the density converges to a constant function as
expected for a Mach number $M\!a \approx \cM^{-2} = 10^{-4}$. We have
omitted the results for the fully discontinuous solution $\uhdg$ for
$\cM = 100$ as the results look similar to
Figure~\ref{fig::mountain_pic_hdg} (which is the result for $\cM =
1$). In particular the magnitude of the velocity error does not
improve with larger $\cM$. This is in accordance to the findings of
the previous example.


\begin{figure}
    \begin{center}
    \includegraphics[height=4cm]{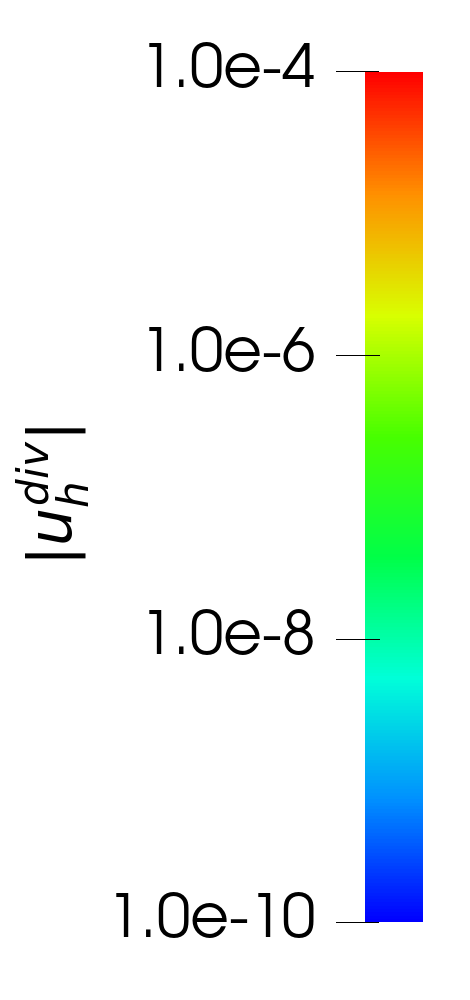}
    \includegraphics[height=4cm]{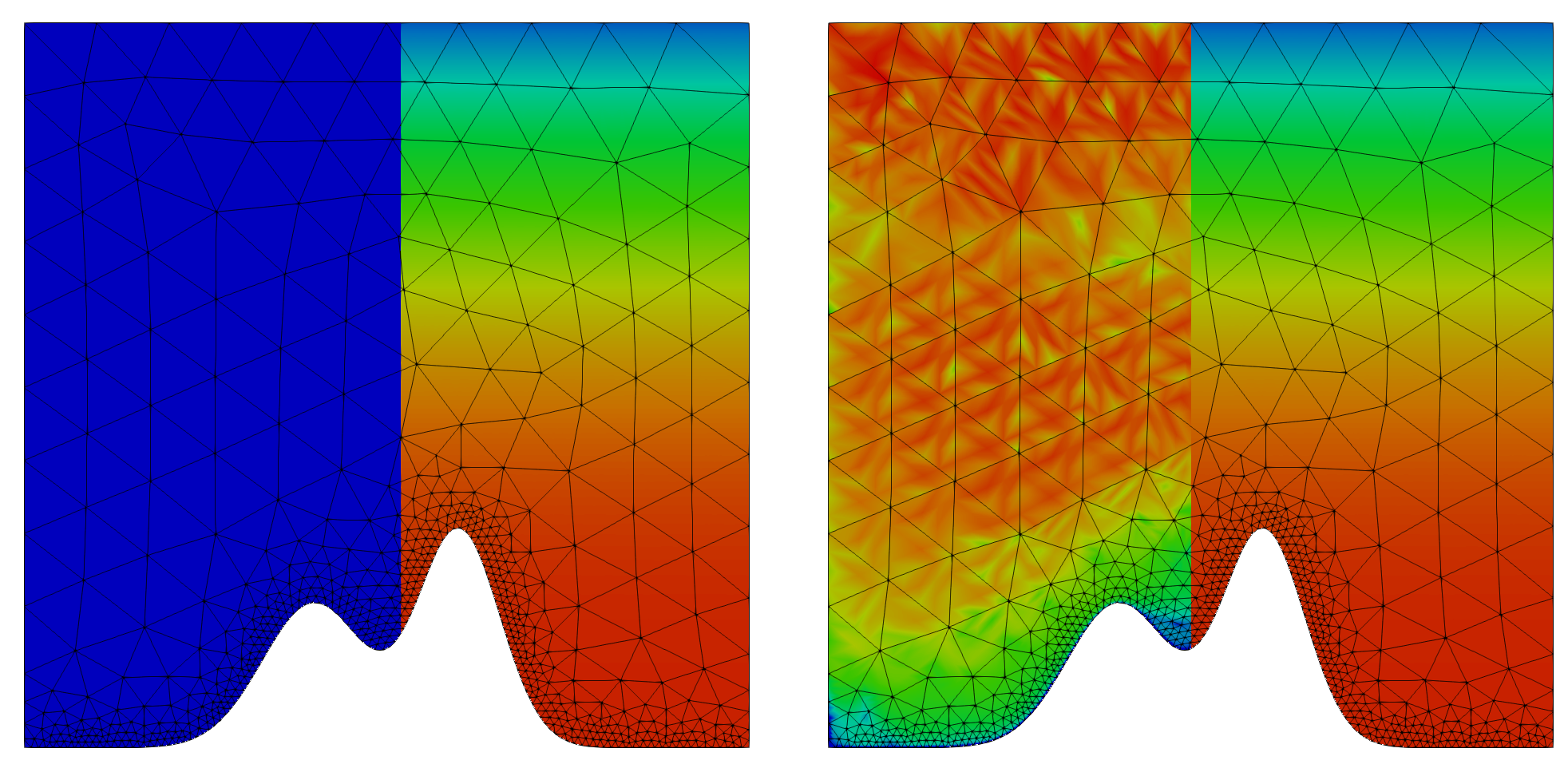}
    \hspace{0.5mm}
    \includegraphics[height=4cm]{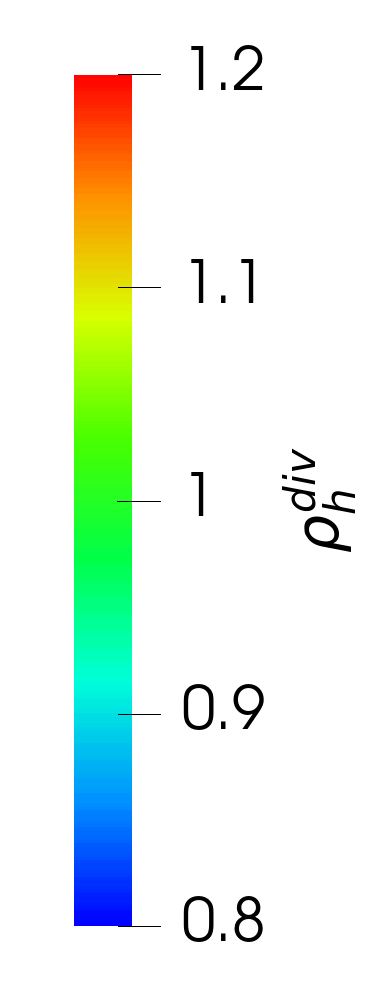}
\end{center}
    \caption{Absolute value of the velocity $| \uhdiv|$ (using a
    logarithmic scale) plotted on the left side of each plot, and
    density $\varrhodiv$ plotted on the right side of each plot, for
    $k = 3$ for the mountain example with $\nu = 1$ and $c_M = 1$
    (left) and $\nu = 10^{-6}$ and $c_M = 1$ (right). }
    \label{fig::mountain_pic_div}
\end{figure}

\begin{figure}
    \begin{center}
    \includegraphics[height=4cm]{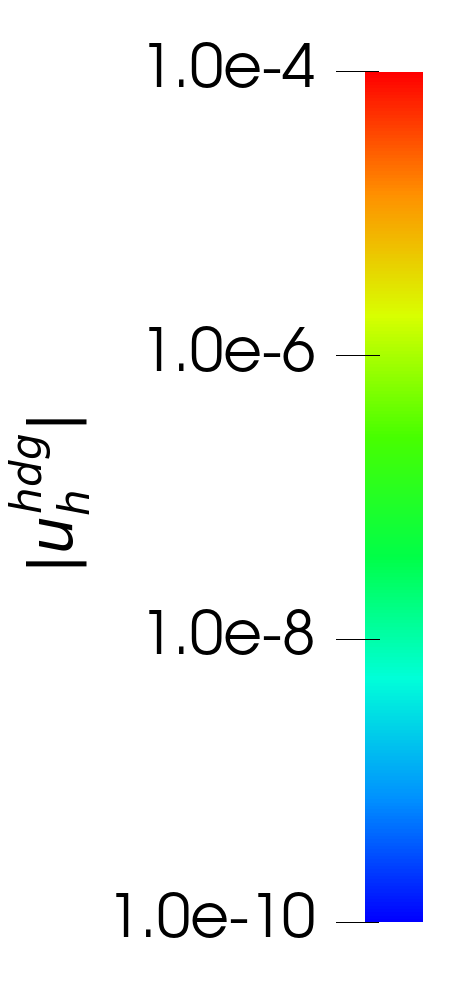}
    \includegraphics[height=4cm]{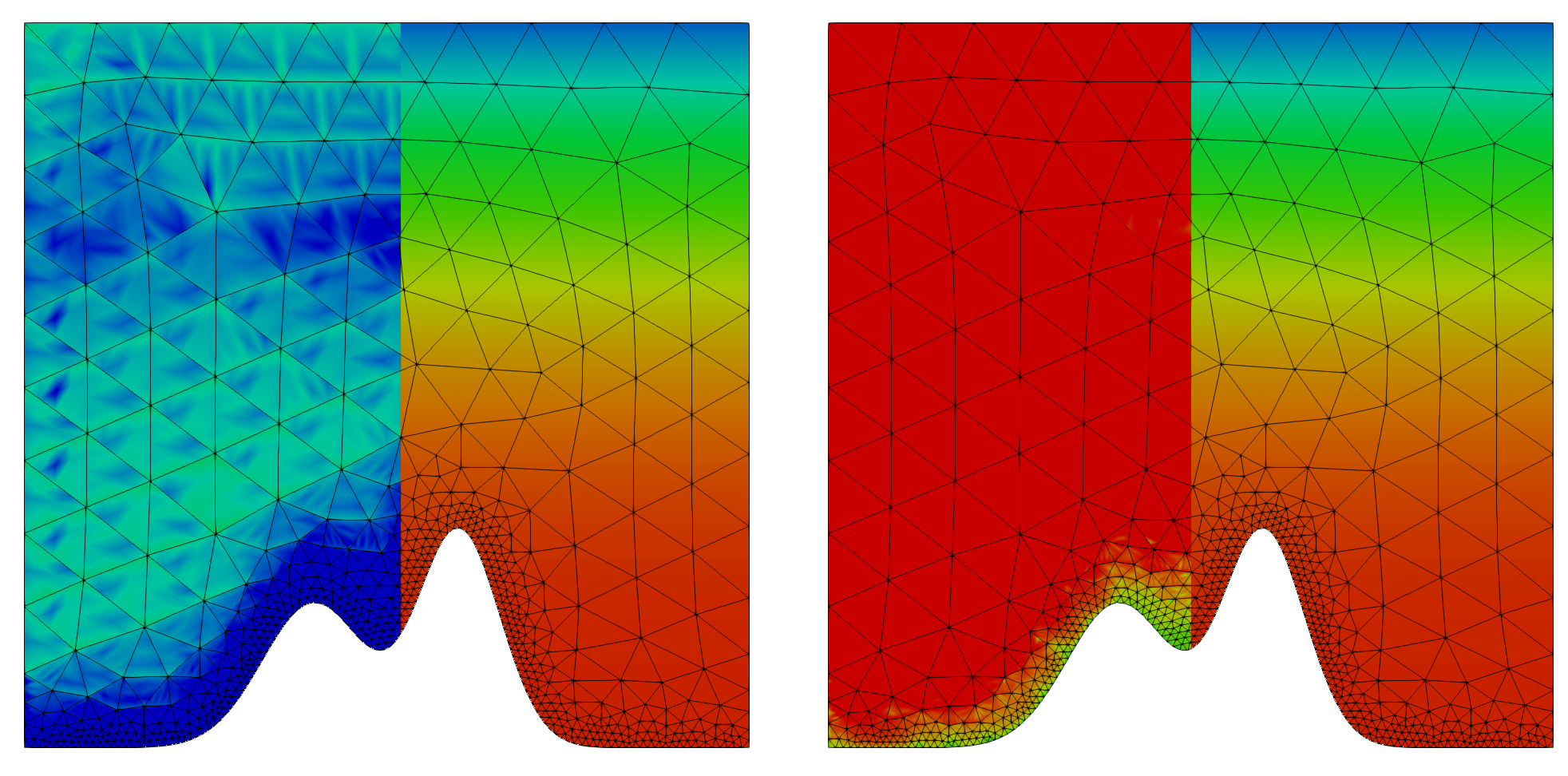}
    \hspace{0.5mm}
    \includegraphics[height=4cm]{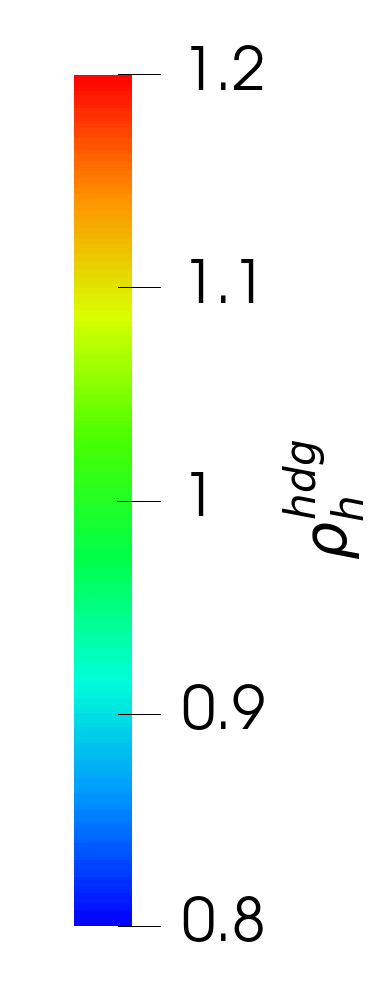}
\end{center}
    \caption{Absolute value of the velocity $| \uhdg|$ (using a
    logarithmic scale) plotted on the left side of each plot, and
    density $\varrhohdg$ plotted on the right side of each plot, for
    $k = 3$ for the mountain example with $\nu = 1$ and $c_M = 1$
    (left) and $\nu = 10^{-6}$ and $c_M = 1$ (right). }
    \label{fig::mountain_pic_hdg}
\end{figure}

\begin{figure}
    \begin{center}
    \includegraphics[height=4cm]{./picture/scale_vel_hdiv_hdg.png}
    \includegraphics[height=4cm]{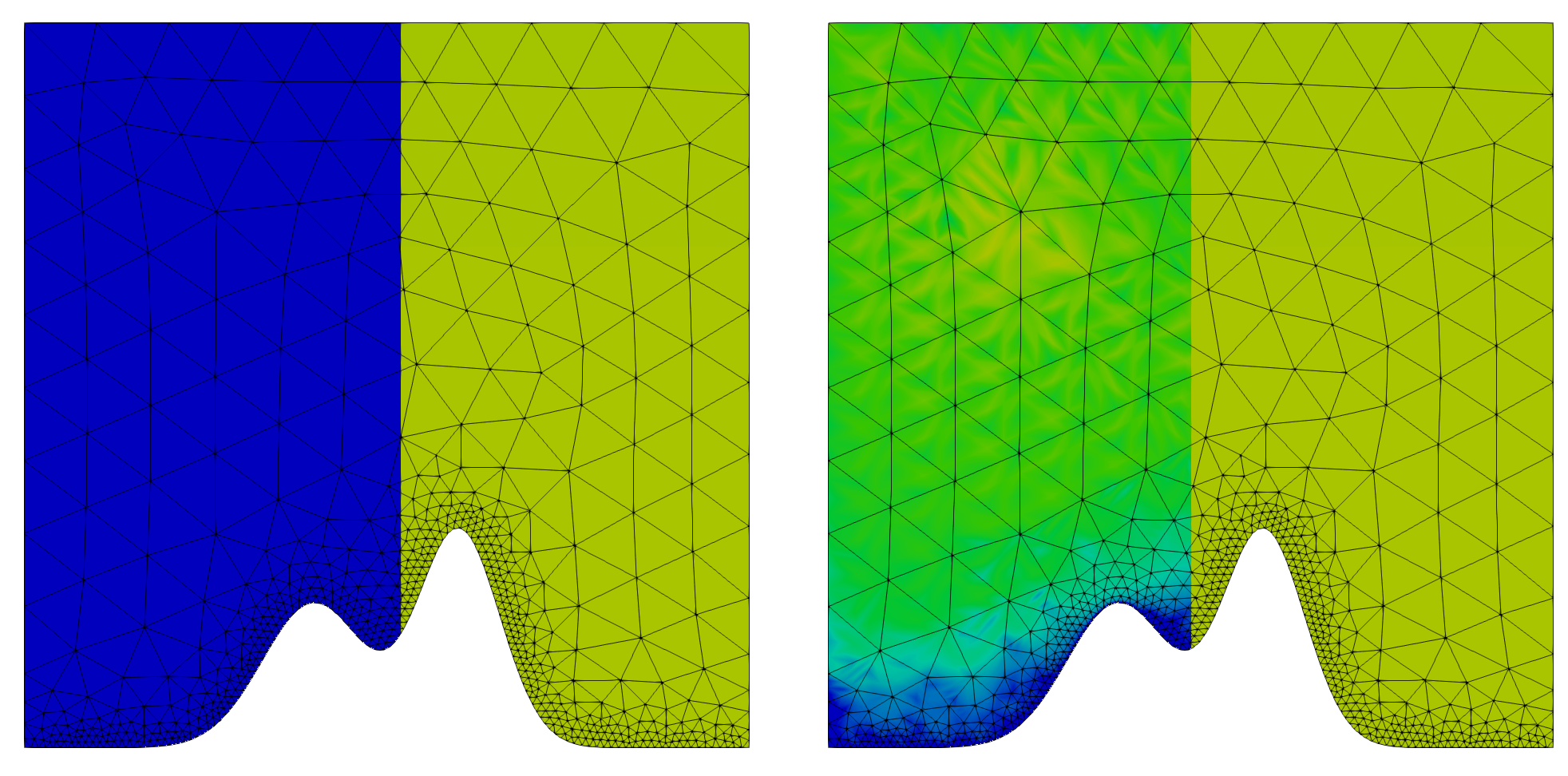}
    \hspace{0.5mm}
    \includegraphics[height=4cm]{./picture/scale_rho_hdiv_hdg.png}
\end{center}
    \caption{Absolute value of the velocity $| \uhdiv|$ (using a
    logarithmic scale) plotted on the left side of each plot, and
    density $\varrhodiv$ plotted on the right side of each plot, for
    $k = 3$ for the mountain example with $\nu = 1$ and $c_M = 100$
    (left) and $\nu = 10^{-6}$ and $c_M = 100$ (right). }
    \label{fig::mountain_pic_div_bigc}
\end{figure}

\begin{figure}
    \includegraphics[width=\textwidth]{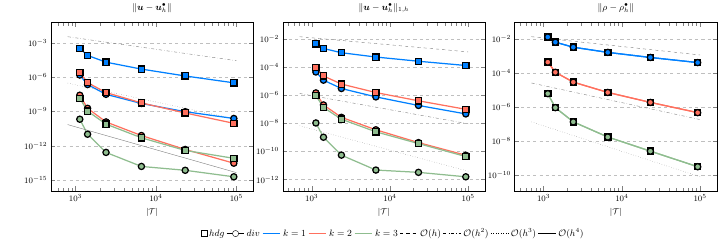}
    \caption{Convergence rates for $k \in \lbrace 1,2,3 \rbrace$ for the mountain example with $\nu = 1$ and $c_M = 1$. }
    \label{fig::mountain_c1_n0}
\end{figure}

\begin{figure}
    \includegraphics[width=\textwidth]{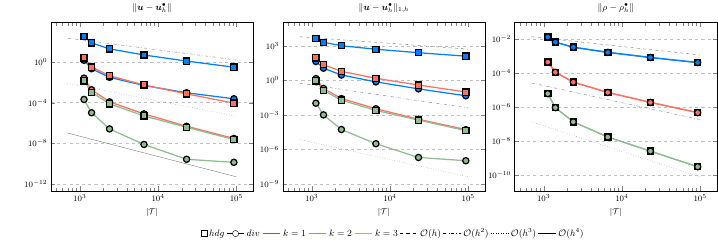}
    \caption{Convergence rates for $k \in \lbrace 1,2,3 \rbrace$ for the mountain example with $\nu = 10^{-6}$ and $c_M = 1$.}
    \label{fig::mountain_c1_n6}
\end{figure}

Finally we present the results if the right hand side is set to
\begin{align} \label{eq::allf_mountain}
    \vecb{g}=0,
    \quad
    \text{and}
    \quad
    \vecb{f} =  \varrho \nabla \Psi,
\end{align}
still with the same solution after \eqref{eq::hydro_numeric}.
Since $\vecb{f}$ is a gradient field, we expect that the normal
continuous solution $\uhdiv$ is not effected by the right hand side
for all viscosities due to its gradient-robustness.  
Indeed, in Figure~\ref{fig::mountain_pic_div_allf} and
Figure~\ref{fig::mountain_pic_hdg_allf} we have again plotted the
solutions. As predicted, the solution $\uhdiv$ is always zero (up to
machine precision) while the non-pressure robust solution $\uhdg$ is
still affected by a decrease of the viscosity and gives the same
results as before (compare to Figure~\ref{fig::mountain_pic_hdg}).

\begin{figure}
    \begin{center}
    \includegraphics[height=4cm]{./picture/scale_vel_hdiv_hdg.png}
    \includegraphics[height=4cm]{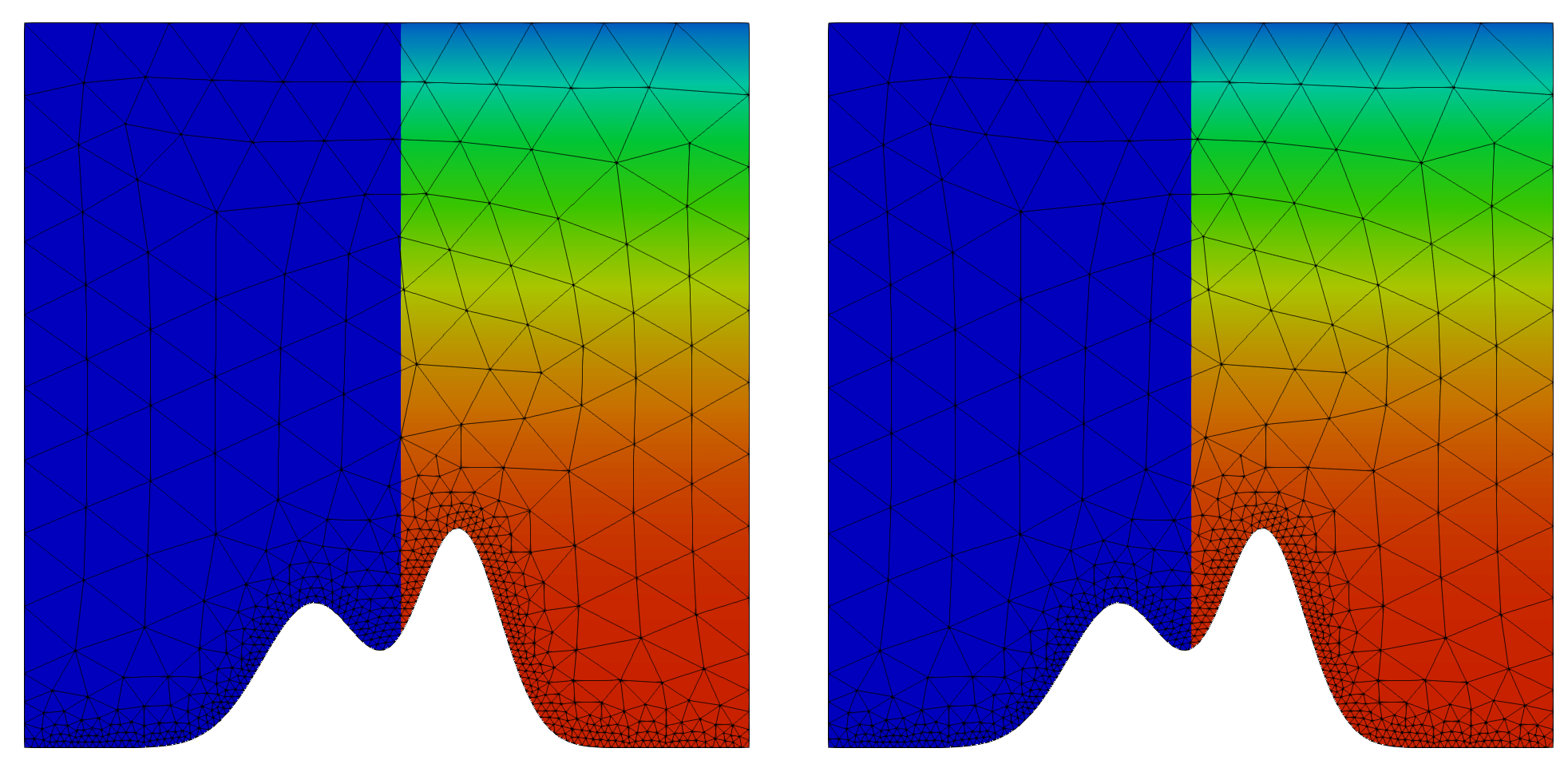}
    \hspace{0.5mm}
    \includegraphics[height=4cm]{./picture/scale_rho_hdiv_hdg.png}
\end{center}
    \caption{Absolute value of the velocity $|\uhdiv|$ (using a
    logarithmic scale) plotted on the left side of each plot, and
    density $\varrhodiv$ plotted on the right side of each plot, for
    $k = 3$ for the mountain example with $\nu = 1$ and $c_M = 1$
    (left) and $\nu = 10^{-6}$ and $c_M = 1$ (right) and the right
    hand side according to \eqref{eq::allf_mountain}. }
    \label{fig::mountain_pic_div_allf}
\end{figure}

\begin{figure}
    \begin{center}
    \includegraphics[height=4cm]{./picture/scale_vel_hdiv_hdg.png}
    \includegraphics[height=4cm]{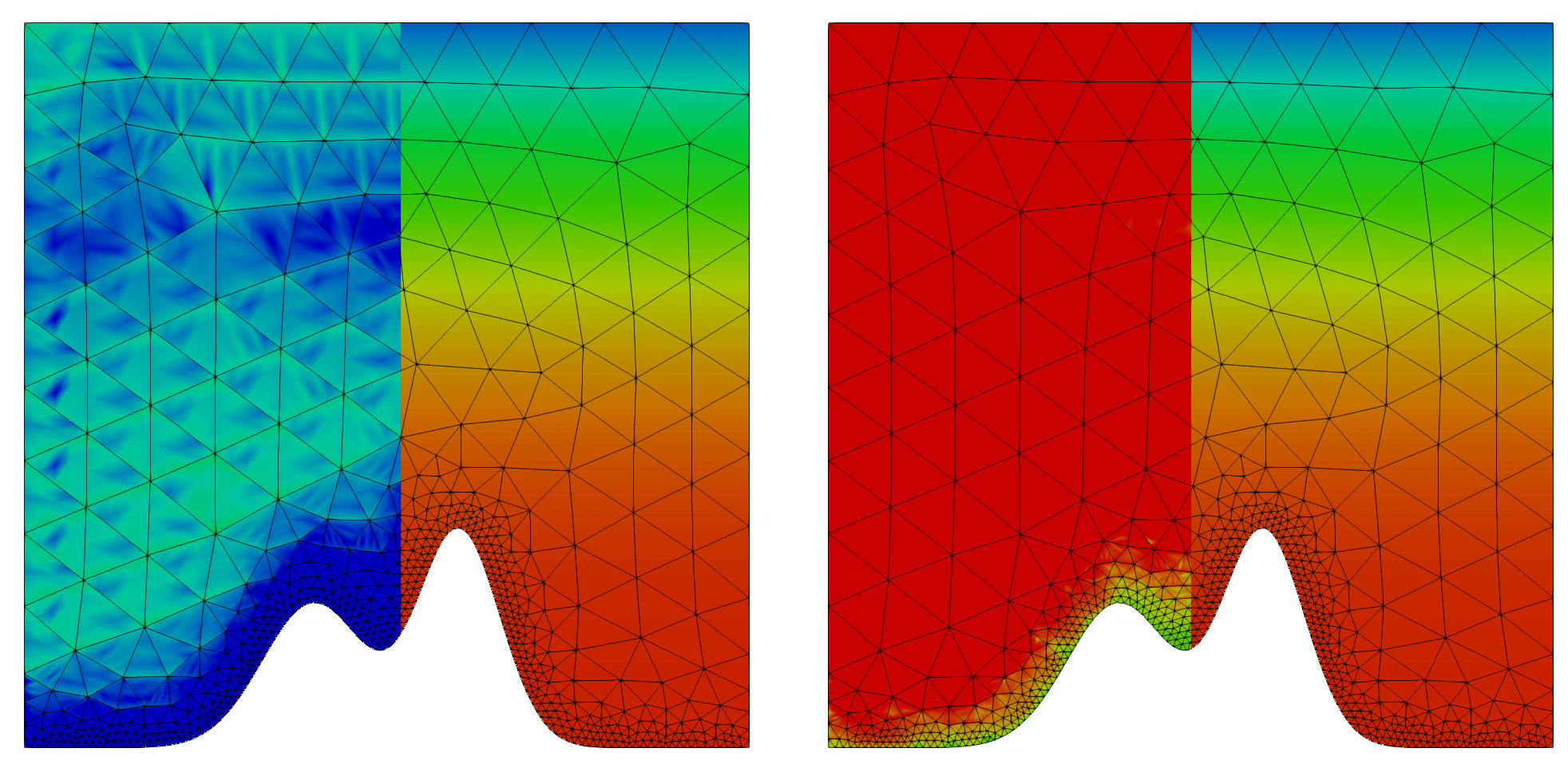}
    \hspace{0.5mm}
    \includegraphics[height=4cm]{./picture/scale_rho_hdiv_hdg.png}
\end{center}
    \caption{Absolute value of the velocity $|\uhdg|$ (using a
    logarithmic scale) plotted on the left side of each plot, and
    density $\varrhohdg$ plotted on the right side of each plot, for
    $k = 3$ for the mountain example with $\nu = 1$ and $c_M = 1$
    (left) and $\nu = 10^{-6}$ and $c_M = 1$ (right) and the right
    hand side according to \eqref{eq::allf_mountain}. }
    \label{fig::mountain_pic_hdg_allf}
\end{figure}

\subsection{A non-hydrostatic well-balanced state}

Gradient-robustness is also relevant, or possibly even more relevant, for the compressible Navier--Stokes problem
\begin{align}\label{eqn:compressible_NSE}
    - \nu \Delta \vecb{u} + \varrho (\vecb{u} \cdot \nabla) \vecb{u} + \nabla p(\varrho) = 0,
    \quad \text{and} \quad
    \mathrm{div}(\varrho \vecb{u}) = 0.
\end{align}
Here, the additional convection term balances the pressure gradient
at least in the limit $\nu \rightarrow 0$ and
non-hydrostatic well-balanced states even in absence of a
gravity term are possible.
A simple example, for any $\nu \geq 0$, can be constructed as follows.
Consider the (divergence-free and harmonic) velocity field
$\vecb{u}(x,y) := (-y,x)^T$
with the convection term
\begin{align*}
    \varrho (\vecb{u} \cdot \nabla) \vecb{u}
    = - \varrho \begin{pmatrix} x \\ y \end{pmatrix}
    = - \varrho \frac{1}{2}\nabla \left(x^2 + y^2\right).
\end{align*}
This term can be balanced by the density $\varrho := \varrho_0 \exp((x^2+y^2)/(2\cM))$, since
\begin{align*}
    \nabla(p(\varrho)) = \cM \varrho \nabla (\log \varrho) = \varrho \frac{1}{2}\nabla \left(x^2 + y^2\right).
\end{align*}
Moreover, the momentum $\varrho\vecb{u}$ is indeed divergence-free due to
\begin{align*}
    \mathrm{div} (\varrho \vecb{u})
    = \nabla \varrho \cdot \vecb{u} + \varrho \mathrm{div} \vecb{u}
    = \frac{\varrho}{\cM} \begin{pmatrix} x \\ y \end{pmatrix} \cdot \begin{pmatrix} -y \\ x \end{pmatrix}  + 0
    = 0.
\end{align*}
Hence, $(\vecb{u}, \varrho)$ is a solution of \eqref{eqn:compressible_NSE}.

To mimic this situation, we solve the compressible Stokes problem with
the right-hand side $\vecb{g}(x,y) := \nabla(x^2 + y^2) / 2$ and
consider as before the cases $\nu \in \lbrace 1,10^{-6} \rbrace$, $\cM \in \lbrace 1,100 \rbrace$ and $k \in
\lbrace 1,2 \rbrace$. Further note, that this example also requires
non-homogeneous boundary data which we prescribe as Dirichlet boundary
conditions for $\vecb{u}$ and, where $\vecb{u} \cdot \vecb{n}$ points
into the domain, as a boundary inflow term for $\varrho$ in the
continuity equation. We can make the following observations:

\begin{itemize}
    \item For the case $\nu = 1$, see Figure~\ref{fig::nvs_c1_n0} for
    $\cM = 1$ and Figure~\ref{fig::nvs_c100_n0} for $\cM = 100$ an
    interesting observation can be made for the lowest order
    approximation. Although the discrete $H^1$-error and the density
    error converge with optimal orders, the $L^2$-error of the
    velocity only gives a linear convergence rate. In contrast to that
    we see, that the quadratic approximation converges optimally, i.e.
    with a cubic rate for the $L^2$-error of the velocity and a
    quadratic rate for all other errors. 
    \item For the case of a vanishing viscosity $\nu = 10^{-6}$, see
    Figure~\ref{fig::nvs_c1_n6} for $\cM = 1$ and
    Figure~\ref{fig::nvs_c100_n6} for $\cM = 100$, we can make the
    same conclusions as for the previous examples. All errors converge
    with optimal order or show some pre-asymptotic faster convergence
    rate, and the gradient-robust solution $\vecb{u}_h^{div}$ provides
    a much better approximation. 
\end{itemize}

\begin{figure}
    \includegraphics[width=\textwidth]{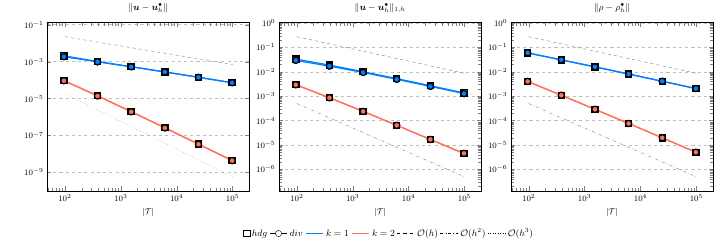}
    \caption{Convergence rates for $k \in \lbrace 1,2 \rbrace$ for the Navier--Stokes
    example with $\nu = 1$ and $c_M = 1$.}
    \label{fig::nvs_c1_n0}
\end{figure}

\begin{figure}
    \includegraphics[width=\textwidth]{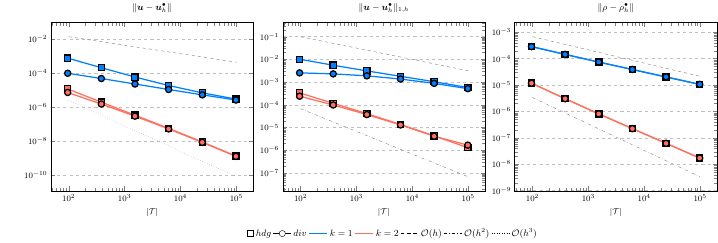}
    \caption{Convergence rates for $k \in \lbrace 1,2 \rbrace$ for the Navier--Stokes 
    example with $\nu = 1$ and $c_M = 100$. }
    \label{fig::nvs_c100_n0}
\end{figure}

\begin{figure}
    \includegraphics[width=\textwidth]{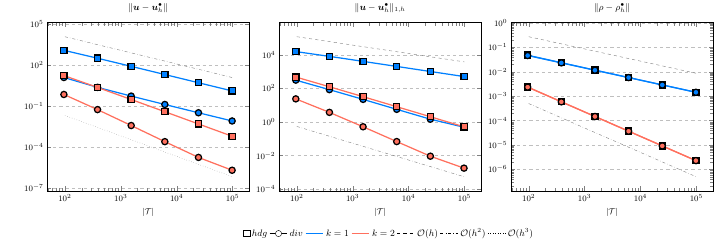}
    \caption{Convergence rates for $k \in \lbrace 1,2 \rbrace$ for the Navier--Stokes
    example with $\nu = 10^{-6}$ and $c_M = 1$. }
    \label{fig::nvs_c1_n6}
\end{figure}

\begin{figure}
    \includegraphics[width=\textwidth]{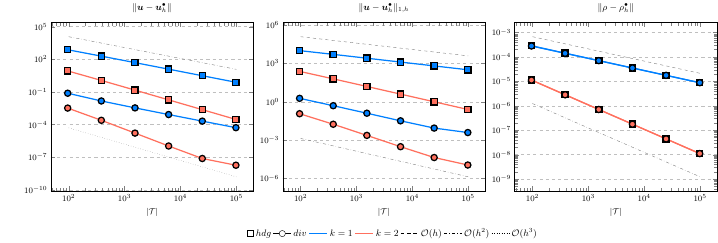}
    \caption{Convergence rates for $k \in \lbrace 1,2 \rbrace$ for the Navier--Stokes
    example with $\nu = 10^{-6}$ and $c_M = 100$. }
    \label{fig::nvs_c100_n6}
\end{figure}

\section{Acknowledgments}
This research was funded in part by the Austrian Science Fund (FWF)
[P35931-N] and the German Science Foundation (DFG) - Project number 467076359.
For the purpose of open access, the author has
applied a CC BY public copyright license to any Author Accepted
Manuscript version arising from this submission.

\bibliographystyle{plain}
\bibliography{lit}
\end{document}